\documentclass[11pt]{article}

\usepackage{amssymb, enumerate}
\usepackage{amsfonts}
\usepackage{amssymb}
\usepackage{amsmath}
\usepackage{color}
\usepackage{hyperref}
\usepackage{tikz}

\definecolor{green3}{rgb}{0,0.6,0}

\newcommand{\R}{\mathbb R}
\newcommand{\C}{\mathbb{C}}
\newcommand{\Z}{\mathbb{Z}}
\newcommand{\N}{\mathbb{N}}
\newcommand{\Q}{\mathbb{Q}}

\newcommand{\conv}{{\rm{conv}}}
\newcommand{\mult}{{\rm{mult}}}
\newcommand{\cA}{\mathcal{A}}
\newcommand{\cB}{\mathcal{B}}

\newcommand{\bff}{\mathbf{f}}
\newcommand{\bfg}{\mathbf{g}}
\newcommand{\bfh}{\mathbf{h}}

\newtheorem{proposition}{Proposition}
\newtheorem{theorem}[proposition]{Theorem}

\newtheorem{remark}[proposition]{Remark}
\newtheorem{corollary}[proposition]{Corollary}
\newtheorem{lemma}[proposition]{Lemma}
\newtheorem{example}{Example}

\newenvironment{proof}{
\trivlist \item[\hskip \labelsep\mbox{\it Proof:
}]}{\hfill\mbox{$\square$}
\endtrivlist}

\title{On the multiplicity of isolated roots of sparse polynomial systems\footnote{Partially supported by the following Argentinian grants: PIP 11220130100527CO  CONICET (2014-2016) and UBACYT 2017, 20020160100039BA.}}
\author{Mar\'\i a Isabel Herrero$^{\sharp,\diamond}$, Gabriela Jeronimo$^{\sharp,\dag,\diamond}$, Juan
Sabia$^{\dag,\diamond}$\\[5mm]
{\small $\sharp$ Departamento de Matem\'atica, Facultad de
Ciencias Exactas y
Naturales,} \\[-1mm] {\small Universidad de Buenos Aires, Ciudad
Universitaria, (1428) Buenos Aires, Argentina}\\[2mm]
{\small $\dag$ Departamento de Ciencias Exactas, Ciclo B\'asico
Com\'un,}\\[-1mm]
{\small Universidad de Buenos Aires, Ciudad Universitaria, (1428)
Buenos Aires, Argentina}\\[2mm]
{\small $\diamond$ IMAS, UBA-CONICET, Buenos Aires, Argentina }}

\begin{document}

\date{}
\maketitle

\begin{abstract}
We give formulas for the multiplicity of any affine isolated zero of a generic polynomial system of $n$ equations in $n$ unknowns with prescribed sets of monomials.
First, we consider sets of supports such that the origin is an isolated root of the corresponding generic system and prove formulas for its multiplicity. Then, we apply these formulas to solve the problem in the general case, by showing that the multiplicity of an arbitrary affine isolated zero of a generic system with given supports equals the multiplicity of the origin as a common zero of a generic system with an associated family of supports.

The formulas obtained are in the spirit of the classical Bernstein's theorem, in the sense that they depend on the combinatorial structure of the system, namely, geometric numerical invariants associated to the supports, such as mixed volumes of convex sets and, alternatively, mixed integrals of convex functions.

\bigskip


\end{abstract}

\section{Introduction}
\label{intro}

The connections between the set of solutions of a polynomial
system and the geometry of the supports of the polynomials
involved have been studied in the literature, starting with the
foundational  work of Bernstein \cite{Ber75}, Kushnirenko
\cite{Kus76} and Khovanskii \cite{Kho78}. They proved that the
number of isolated solutions in $(\C^*)^n$ (where $\C^*:=\C
\setminus \{0\}$) of a system with $n$ polynomial equations in $n$
unknowns is bounded from above by the mixed volume of their
support sets. Afterwards, combinatorial invariants of the same
type also allowed to obtain bounds for the number of isolated
solutions of the system in the affine space $\C^n$ (see, for
example, \cite{Roj94}, \cite{RW96}, \cite{LW96}, \cite{HS97} and
\cite{EV99}). In \cite{PS08b}, another refinement of Bernstein's
bound was given by introducing mixed integrals of concave
functions to estimate the number of isolated solutions in
$\C\times (\C^*)^{n-1}$. Concerning algorithmic complexity,
counting the number of isolated roots is known to be
$\#\mathbf{P}$-complete already for binomial systems \cite{CD07}.
The complexity of counting irreducible components of algebraic
varieties is studied in \cite{BS09}.

Even though the common zeroes of sparse polynomial systems in
$(\C^*)^n$ are generically simple, the isolated roots on
coordinate hyperplanes may generically have high multiplicity. The
aim of this paper is to prove formulas for the multiplicity of the
isolated affine zeroes of generic sparse polynomial systems in
terms of the geometry of their supports. This dependence is
already present in the seminal work of Kushnirenko \cite{Kou76},
where the Milnor number of the singularity at the origin of a
hypersurface is studied.

Several authors have used geometric tools, including convex sets, volumes and covolumes, to solve related problems. Geometric invariants of this type are considered in \cite{Tei88}  to determine  multiplicities of monomial ideals in local rings. In \cite[Chapter 5]{GKZ94}, the multiplicity of a singular point on a toric variety is given as a normalized volume. A particular case of this result is recovered in \cite{CD06}, where the multiplicity of the origin as an isolated zero of a generic unmixed polynomial system is computed under the assumption that each polynomial contains a pure power of each variable.  A generalization of this result to  the mixed case under the same assumption can be found in \cite{KK14}, where the multiplicity of the origin is expressed in terms of mixed covolumes. Recently, in \cite{Mondal16} a formula for the intersection multiplicity at the origin of the hypersurfaces defined in $\C^n$ by $n$ generic polynomials with fixed Newton diagrams is proved.

In this paper, we obtain formulas for the multiplicities of all the affine isolated zeros of a generic polynomial system of $n$ polynomials in $n$ variables with given supports in terms of mixed volumes and, alternatively, in terms of mixed integrals of convex functions associated to the supports of the polynomials involved (see Theorem \ref{teo:MultXiTODO} in Section \ref{sec:multaffine} below).

First, we consider the case of the origin as an isolated zero of a generic system where each polynomial contains a pure power of each variable (see Theorem \ref{teo:MIcasoTocandoEjes} in Section \ref{section:toca ejes}). A formula for the multiplicity of the origin under this particular hypothesis has already been obtained in \cite{KK14} in terms of different invariants. Then, we analyze the case of generic systems with arbitrary supports such that the origin is an isolated zero  (see Proposition \ref{prop:Mgrande} and Corollary \ref{cor:mult0MI} in Section \ref{subsec:generalcase}).

Finally, in order to deal with arbitrary affine isolated zeros, the result in \cite[Proposition 6]{HJS13} enables us to determine all sets $I\subset \{1,\dots, n\}$ such that a generic system with the given supports has isolated zeros whose vanishing coordinates are indexed by $I$.
For such an isolated zero, we prove that its multiplicity equals the multiplicity of the origin as an isolated zero of an associated generic sparse system of $\#I$ polynomials in $\#I$ variables whose supports can be explicitly defined from the input supports and the set $I$ (see Theorem \ref{teo:multF=multG} in Section \ref{sec:multaffine}). Thus, a formula for the multiplicity of an arbitrary affine zero of the system follows from our previous result concerning the multiplicity of the origin.

Our formulas for the multiplicity of the origin can be seen as a generalization of those in \cite{KK14}, in the sense that the only hypotheses on the supports we make are the necessary ones, proved in \cite[Proposition 6]{HJS13}, so that the origin is an isolated zero of a generic system with the given supports.  An earlier approach from \cite{Mondal16} to compute the multiplicity of the origin under no further assumptions on the supports leads to a formula which, unlike ours, is not symmetric in the input polynomials, as already stated by the author. Furthermore, in this paper we give formulas for the multiplicity of arbitrary isolated affine zeroes of a generic sparse system.

The paper is organized as follows: Section \ref{sec:prelim}  recalls the definitions and basic properties of mixed volumes and mixed integrals, and describes the algorithmic approach to compute multiplicities of isolated zeros of polynomial systems by means of basic linear algebra given in \cite{DZ05}, which we use as a tool. In Section \ref{sec:mult0}, formulas for the multiplicity of the origin are obtained, first for systems where each polynomial contains a pure power of each variable and then, in the general case. Finally, Section \ref{sec:multother} is devoted to computing the multiplicity of an arbitrary affine isolated zero of a generic system.

\section{Preliminaries} \label{sec:prelim}

\subsection{Mixed volume and stable mixed volume}\label{subsec:SMV}

Let $\cA_1,\dots, \cA_n$ be finite subsets of $(\Z_{\ge 0})^n$.  A \emph{sparse polynomial system supported
on $\cA = (\cA_1,\dots, \cA_n)$} is given by polynomials
$$f_j =
\sum_{a\in \cA_j} c_{j,a} x^a$$ in the variables $x = (x_1,\dots, x_n)$, with $c_{j,a} \in \C \setminus \{0\}$
for each $ a \in \cA_j$ and $1\le  j \le n$.

We denote by $MV_n(\cA) = MV_n(\cA_1,\dots, \cA_n)$ the \emph{mixed volume} of
the convex hulls of $\cA_1,\ldots,\cA_n$ in $\R^{n}$, which is defined as
$$MV_n(\cA)= \sum_{J\subset \{ 1,\dots, n\}} (-1)^{n-\#J}\, Vol_{n} \Big(\sum_{j\in J} \text{conv}(\cA_j)\Big)$$
 (see, for example, \cite[Chapter 7]{CLO}). The mixed volume of $\cA$ is an upper
bound for the number of isolated roots in $(\C^*)^n$ of a sparse
system supported on $\cA$  (see \cite{Ber75}).

The \emph{stable mixed volume} of $\cA=(\cA_1,\dots, \cA_n)$, denoted by
$SM_n(\cA) = SM_n(\cA_1,\dots, \cA_n)$, is introduced in \cite{HS97} to estimate the number of isolated roots in $\C^n$ of a sparse polynomial system supported on
$\cA$ and is defined as follows. Let $\cA^{0} = (\cA_1^{0}, \dots, \cA_n^{0})$ be the family with
$\cA_j^{0} := \cA_j \cup \{0\}$ for every $1\le j \le n$, and let  $\omega^{0} =
(\omega^{0}_1, \dots, \omega^{0}_n)$ be the lifting function for $\cA^{0}$
defined by $\omega^{0}_j(q) = 0$ if $q \in \cA_j$ and
$\omega^{0}_j(0)= 1$ if $0 \notin \cA_j$.
Consider the polytope $Q^0$ in $\R^{n+1}$ obtained by taking the Minkowski (pointwise) sum of the convex hulls of the graphs of $\omega_1^0, \dots, \omega_n^0$. The projection of the lower facets of $Q^0$ (that is, the $n$-dimensional faces with inner normal vector with a positive last coordinate) induces a subdivision of $\cA^0$. A cell $C= (C_1,\dots, C_n)$, with $C_j \subset \cA_j^0$ for every $1\le j \le n$, of this subdivision is said to be \emph{stable} if it corresponds to a facet of $Q^0$ having an inner normal vector with all non-negative coordinates.
The stable mixed volume $SM_n(\cA_1,\dots, \cA_n)$ is  the sum of the mixed volumes of all the stable cells in the subdivision of $\cA^{0}$.

Note that $\cA= (\cA_1,\dots, \cA_n)$ is a stable cell in the defined subdivision of $\cA^0$, namely, the cell with associated inner normal vector $(0,\dots, 0,1)$; therefore, we have that
$$MV_n(\cA_1, \dots, \cA_n)\le SM_n(\cA_1, \dots, \cA_n)\le
MV_n(\cA_1\cup\{0\}, \dots, \cA_n\cup\{0\}).$$

\subsection{Mixed integrals for concave and convex functions}
\label{section:general}

Let $P_1,\dots, P_n$ be polytopes in $\R^{n-1}$, and, for $1\le j \le n$, let $\sigma_j: P_j\to \R$ be a concave function and $\rho_j: P_j \to \R$ a convex function.
Following \cite{PS08a}, we can define concave (respectively convex)
functions as:
$$\begin{array}{l}\sigma_i\boxplus\sigma_j: P_i+P_j \to \R, \\
\sigma_i\boxplus\sigma_j(x)=\max\{ \sigma_i(y)+\sigma_j(z)  :  y
\in P_i,\, z \in P_j,\, y+z=x\}\end{array}
$$
and
$$\begin{array}{l} \rho_i\boxplus'\rho_j: P_i+P_j \to \R, \\
\rho_i\boxplus'\rho_j(x)=\min\{ \rho_i(y)+\rho_j(z)  :  y
\in P_i,\, z \in P_j,\, y+z=x\}.\end{array}$$
Note that $\rho_i\boxplus'\rho_j=
-(-\rho_i)\boxplus(-\rho_j)$.

In the same way, for every non-empty subset $J\subset\{ 1,\dots, n\}$, we can define $$\boxplus_{j\in J}\sigma_j: \sum_{j\in J} P_j  \to \R \quad \hbox{ and }\quad \boxplus'_{j\in J}\rho_j: \sum_{j\in J} P_j  \to \R.$$
The mixed integrals of $\sigma_1, \dots, \sigma_n$ (respectively, $\rho_1, \dots, \rho_n$)
are defined as:
$$MI_n(\sigma_1, \dots, \sigma_n)= \sum_{k=1}^n(-1)^{n-k}
\sum_{J \subset \{1, \dots, n\} \atop \#J=k} \int_{\sum_{j \in
J}P_j}\boxplus_{j\in J} \sigma_j(x)dx, $$
$$MI_n'(\rho_1, \dots, \rho_n)= \sum_{k=1}^n(-1)^{n-k}
\sum_{J \subset \{1, \dots, n\} \atop \#J=k} \int_{\sum_{j \in
J}P_j}\boxplus'_{j\in J}\rho_j(x)dx.$$

For a polytope $P\subset \R^{n-1}$, a convex function $\rho : P \to \R$ and a concave function $\sigma: P \to \R$ such that $\rho(x) \le \sigma(x)$ for every $x\in P$, we denote
$$P_{\rho, \sigma} = \mbox{conv}(\{(x, \rho(x)) : x\in P\} \cup \{(x, \sigma(x)) : x\in P\}).$$

Given a polytope $Q\subset \R^n$, if $\pi:\R^n \rightarrow \R^{n-1}$ is the projection to the first $n-1$ coordinates, we may define a concave function
$\sigma_Q:\pi(Q) \rightarrow \R$ and a convex function
$\rho_Q:\pi(Q) \rightarrow \R$  as:
$$ \sigma_Q(x)=\max\{x_n \in \R : (x,x_n) \in Q\} \ \hbox{ and } \
\rho_Q(x)=\min\{x_n \in \R : (x,x_n) \in Q\}.$$

\begin{remark} \label{rem:def x piso y techo} The functions $\sigma_Q$ and $\rho_Q$ defined above parameterize the lower and
upper envelopes of $Q$ respectively. Moreover, $\pi(Q)_{\rho_Q, \sigma_Q} = Q$.
\end{remark}

Let $Q_1,\dots, Q_n$ be polytopes in $\R^n$. For $1\le j \le n$, let $\sigma_j = \sigma_{Q_j}$ and $\rho_j = \rho_{Q_j}$. Let $J\subset \{1,\dots, n\}$, $J \ne \emptyset$. Then, $\boxplus_{j\in J} \sigma_j:\sum_{j\in J}\pi(Q_j) \to \R $ and $\boxplus'_{j\in J} \rho_j: \sum_{j\in J}\pi(Q_j) \to \R$ parameterize the upper and lower envelopes of $\sum_{j\in J} Q_j$ respectively.

\subsection{Multiplicity matrices}\label{sec:DZ}

In order to compute multiplicities of isolated zeros of polynomial systems, we will follow the algorithmic approach from \cite{DZ05} based on duality theory, which we briefly recall in this section.

Let $\bff = (f_1,\dots, f_n)$ be a system of polynomials in $\C[x_1,\dots, x_n]$. Denote $\mathcal{I}$ the ideal of $\C[x]=\C[x_1,\dots, x_n]$ generated by $f_1, \dots, f_n$.

For an isolated zero $\zeta\in \C^n$  of the system $\bff$, we denote $\mult_\zeta(\bff)$ its multiplicity, defined as the dimension (as a $\C$-vector space) of the local ring $\C[x]_{\mathfrak{m}_\zeta}/\mathcal{I}\C[x]_{\mathfrak{m}_\zeta}$, where $\mathfrak{m}_\zeta = (x_1 - \zeta_1,\dots, x_n- \zeta_n)$ is the maximal ideal associated with $\zeta$ (see, for instance, \cite[Chapter 4, Definition (2.1)]{CLO}).

Let $\mathcal{D}_\zeta(\mathcal{I})$ the dual space of the ideal $\mathcal{I}$ at $\zeta$; namely, the vector space
$$\mathcal{D}_\zeta(\mathcal{I})= \Big\{c=\sum_{\alpha \in (\Z_{\ge 0})^n}c_\alpha\, \partial_\alpha[\zeta] \mid  c(f)=0 \mbox{ for all } f \in \mathcal{I}\Big\},$$ where, for every $\alpha =(\alpha_1,\dots, \alpha_n)\in (\Z_{\ge 0})^n$,  $c_\alpha\in \C$,
\begin{equation}\label{eq:deralpha}
\partial_\alpha = \dfrac{1}{\alpha_1! \dots \alpha_n!} \dfrac{\partial^{|\alpha|}}{\partial x_1^{\alpha_1}\dots, x_n^{\alpha_n}},
\end{equation}
and
$$\partial_\alpha[\zeta] : \C[x] \to \C, \quad \partial_\alpha[\zeta](f) = (\partial_\alpha f)(\zeta).$$
The dimension of $\mathcal{D}_\zeta(\mathcal{I})$ equals the multiplicity of $\zeta$ as a zero of $\mathcal{I}$ (see \cite{Macaulay1916}, \cite{Stetter04}).

For
every $k \ge 0$, consider the subspace
$$\mathcal{D}_\zeta^k(\mathcal{I})=\Big\{c =\sum_{\alpha \in (\Z_{\ge 0})^n, \ |\alpha| \le k }c_\alpha\, \partial_\alpha [\zeta] \mid  c(f)=0 \mbox{ for all } f \in \mathcal{I}\Big\}$$
of all functionals in $\mathcal{D}_\zeta(\mathcal{I})$ with differential order bounded by $k$.
Since $\zeta$ is an isolated common zero of $\mathcal{I}$, there exists $k_0 \in
\Z_{\ge 0}$ such that $\mathcal{D}_\zeta(\mathcal{I})=\mathcal{D}_\zeta^{k_0}(\mathcal{I}) =
\mathcal{D}_\zeta^k(\mathcal{I})$ for all $k\ge k_0$ and $\dim (\mathcal{D}_\zeta^k(\mathcal{I}))<\dim(\mathcal{D}_\zeta^{k+1}(\mathcal{I}))$ for every $0 \le k < k_0$ (see \cite[Lemma 1]{DZ05}).

Following \cite[Section 4]{DZ05}, the dimension of the vector spaces $\mathcal{D}_\zeta^k(\mathcal{I})$ can be
computed by means of the \emph{multiplicity matrices}, defined as follows. For $k=0$, set $S_0(\bff, \zeta) =[f_1(\xi) \cdots f_n(\xi)]^t = 0\in \C^{n\times 1}$. Take $\prec$ a graded monomial ordering. For $k\ge 1$, consider the sets
$\mathbb{I}_k= \{\alpha \in
(\Z_{\ge0})^n \mid |\alpha| \le k\}$ ordered by $\prec$, and $\mathbb{I}_{k-1}\times\{1, \dots, n\}$ with the ordering $(\beta,j)\prec (\beta',j')$ if $ \beta\prec \beta'$ or
$\beta=\beta'$ and $j<j'$.
Let $S_k(\bff, \zeta)$ be the $
\binom{k-1+n}{k-1}n \times \binom{k+n}{k}$ matrix whose columns
are indexed by $\mathbb{I}_k$ (corresponding to the differential functionals $\partial_\alpha$ for
$\alpha \in \mathbb{I}_k$) and whose rows are indexed by  $(\beta,j)\in \mathbb{I}_{k-1}\times\{1, \dots, n\}$ (corresponding to the polynomials $(x-\zeta)^\beta f_j$) such that the entry at the intersection of the row indexed by $(\beta, j)$ and the column indexed by $\alpha$ is
$$(S_k(\bff, \zeta))_{(\beta, j), \alpha} = \partial_\alpha((x-\zeta)^\beta f_j)(\zeta).$$
(Here, $(x-\zeta)^\beta = (x_1 - \zeta_1)^{\beta_1}\cdots (x_n-\zeta_n)^{\beta_n}$.)
Then, the dimension of
$\mathcal{D}_\zeta^k(\mathcal{I})$ equals the dimension of the nullspace of $S_k(\bff, \zeta)$
(see \cite[Theorems 1 and 2]{DZ05}).
As a consequence:

\begin{proposition} With the previous assumptions and notation,  if $$k_0= \min\{ k\in \Z_{\ge 0} \mid \dim(\ker(S_k(\bff, \zeta))) = \dim(\ker(S_{k+1}(\bff, \zeta)))\},$$ the multiplicity of $\zeta$ as an isolated zero of $\bff$ is $\mult_\zeta(\bff) = \dim(\ker(S_{k}(\bff, \zeta)))$ for any $k\ge k_0$.
\end{proposition}

\section{Multiplicity of the origin} \label{sec:mult0}

Consider a family $\cA=(\cA_1, \dots, \cA_n)$ of finite sets in
$(\Z_{\ge 0})^n$ such that $0 \not\in \cA_j $ for all $ 1 \le j \le n$.
Under this assumption, $0\in \C^n$ is a common zero of any sparse system of polynomials $f_1,\dots,
f_n\in \C[x_1, \dots, x_n]$ supported on $\cA$.

We are interested in the case when $0$ is an \emph{isolated} common zero of the system.
By \cite[Proposition 6]{HJS13}, for a \emph{generic} family of polynomials $\mathbf{f} = f_1,\dots,
f_n\in \C[x_1, \dots, x_n]$ supported on $\cA$, we have that $0$ is an isolated point of $V(\mathbf{f})$ if and only if
$ \#I+\#J_I\ge n $ for all $I \subset \{1, \dots,
n\},$ where $J_I$ is the set of subindexes of all polynomials
that do not vanish when we evaluate $x_i=0$ for all $i \in I$.

Every $\mathbf{c} = (\mathbf{c}_1, \dots, \mathbf{c}_n)\in \C^{\#\cA_1}\times \dots \times \C^{\#\cA_n}$ defines a system $\bff_{\mathbf{c}}$ of polynomials with coefficients $\mathbf{c}$ supported on a family of subsets of $\cA_1,\dots, \cA_n$. If $0$ is an isolated zero of $\bff_{\mathbf{c}}$, we define $\mult_{\cA}(\mathbf{c}) :=\mult_0(\bff_{\mathbf{c}})\in \Z_{>0}$.

\begin{lemma}\label{lem:multgen} Under the previous assumptions and notation, let $\mu_\cA$ be the minimum of the function $\mult_{\cA}$. Then,
$\{ \mathbf{c}\in \C^{\#\cA_1}\times \dots \times \C^{\#\cA_n} \mid \mult_{\cA}(\mathbf{c})= \mu_\cA\}$ contains a non-empty Zariski open set of $\C^{\#\cA_1}\times \dots \times \C^{\#\cA_n}$.
\end{lemma}

\begin{proof}
It is straightforward, for example, from the computation of multiplicities by using multiplicity matrices (see Section \ref{sec:DZ}).
\end{proof}

In this sense, we may speak of $\mu_\cA$ as the multiplicity  of $0$ as an isolated root of a \emph{generic} sparse system supported on $\cA$. Explicit conditions on the coefficients satisfying $\mult_{\cA}(\mathbf{c})= \mu_\cA$ are given in \cite[Theorem 4.12]{Mondal16}.

\bigskip

Therefore, in this section, we will focus on the computation of the multiplicity of the origin
as a common zero of a generic polynomial system supported on $\cA =(\cA_1,\dots, \cA_n)$, under the following assumptions:
\begin{itemize}
\item[(H1)] $0\notin \cA_j \subset (\Z_{\ge 0})^n$ for every $1\le j \le n$;
\item[(H2)] for all $I \subset \{1,\dots, n\}$, if $J_I:= \{ j\in \{1,\dots, n\} \mid \exists a\in \cA_j : a_i = 0 \, \forall i\in I\}$, then $ \#I+\#J_I\ge n $.
\end{itemize}

Moreover, in \cite[Proposition 5]{HJS13}, these conditions are proved to be equivalent to the fact that,  for a generic system $\mathbf{f}$ supported on $\cA$ and vanishing at $0\in \C^n$, the variety $V(\mathbf{f})$ consists only of isolated points in $\C^n$.

Under these assumptions, by \cite[Theorem 2]{HS97}, the number of common zeros of $\bff$ in $\C^n$ counted with multiplicities is the stable mixed volume
$SM_n(\cA)$.
In particular, since the number of common zeros of the system in $(\C^*)^n$ is the mixed volume $MV_n(\cA)$ (see \cite{Ber75}), we have that
\begin{equation} \label{eq:cotamult}
\mult_0(\bff) \le SM_n(\cA) - MV_n(\cA) \le MV_n(\cA^0) - MV_n(\cA),
\end{equation}
where $\cA^0 = (\cA_1\cup \{0\},\dots, \cA_n \cup \{ 0 \})$.

\subsection{A particular case}\label{section:toca ejes}

The first case we are going to consider is when the following stronger assumption on $\cA$ holds:

\begin{itemize}
\item[(H3)] For every $ 1
\le i, j \le n$, there exists $\mu_{ij} \in \N$ such that $\mu_{ij}e_i \in \cA_j$, where $e_i $ is the $i$th vector of the canonical basis of $\Q^n$.
\end{itemize}

Note that assumption (H3) implies that assumption (H2) holds.

\medskip

Under condition (H3), in \cite[Theorem 7.6]{KK14} the multiplicity of the origin as an isolated common zero
of a generic  polynomial system supported on $\cA$ is computed in terms of covolumes of coconvex bodies associated
to $\cA$.
Here, we will first re-obtain this result by  proving a formula using mixed volumes of convex polytopes and then, we will reformulate this formula in terms of mixed integrals of convex functions.

\medskip

We start by comparing stable mixed volumes with mixed volumes in our particular setting.

\begin{lemma} \label{lem:SMvsMV0} With the previous notation, if assumptions (H1) and (H3) hold, we have that $$SM_n(\cA_1, \dots,
\cA_n)=MV_n(\cA_1^0, \dots, \cA_n^0).$$
\end{lemma}

\begin{proof} It suffices to prove that every cell in the subdivision of $\cA^0= (\cA_1^0, \dots, \cA_n^0)$ induced by the lifting function introduced in Section \ref{subsec:SMV} is stable.

Consider a cell $C=(C_1,\dots,C_n)$ of the stated subdivision different from
$(\cA_1,\dots, \cA_n)$ (for which the result is trivial), and let
$\eta=(\eta_1, \dots, \eta_n,1)$ be its associated inner normal vector. We have to show that $\eta_i\ge 0$ for every $1\le i \le n$.

For every $1 \le j \le n$, there exists $a_{C_j}\in \R$ such that
$a_{C_j}= \eta .\, (q, \omega_j^0(q))$ for all $ q \in {C}_j$
and $a_{C_j}\le  \eta .\, (q, \omega_j^0(q))$ for all $ q \in
\cA_j^0$. As the cell $C$ is not $(\cA_1,\dots, \cA_n)$, there
exists $j_0$ such that $0\in C_{j_0}$ and $0\notin \cA_{j_0}$;
then, $a_{C_{j_0}}=\eta .\, (0,1) =1$. Since, by assumption
(H3), for all $1 \le i \le n$, there exists $\mu_{ij_0} \in \N$
such that $\mu_{i j_0} e_i \in \cA_{j_0}^0$, then,
$1=a_{C_{j_0}}\le \eta\, .\, \mu_{ij_0}(e_i,0)=\eta_i\mu_{ij_0}
$. The result follows from the fact that $\mu_{ij_0}>0$ for all
$1\le i\le n$.
\end{proof}

Now, we can state our first formula for the multiplicity of the origin.

\begin{proposition}\label{prop:mult0}  Let $\bff = (f_1, \dots, f_n)$ be a generic polynomial system in $\C[x_1, \dots, x_n]$  supported on a family  $\cA= (\cA_1, \dots, \cA_n)$ of finite sets of $(\Z_{\ge 0})^n$ satisfying assumptions (H1) and (H3). Then, the origin is an isolated common zero of $\bff$ and
$$\mult_0 (\bff) = MV_n(\cA^{0})- MV_n(\cA).$$
\end{proposition}

\begin{proof}
Assumption (H1) implies that the origin is a common zero of the polynomials $\bff$.
In addition, by assumption (H3), the only common zero of $\bff$ not in $(\C^*)^n$
is the origin. Then, all the common zeros of $\bff$ in $\C^n$ are isolated and so, the number of these common zeros is $SM_n(\cA)$ (see \cite{HS97}).
Finally, since the number of common zeros of $\bff$ in $(\C^*)^n$ is $MV_n(\cA)$ (see \cite{Ber75}) and all these zeros have multiplicity $1$ (see \cite{Oka97}), we deduce that $MV_n(\cA) + \mult_0(\bff) = SM_n(\cA)$. Thus, the result follows from Lemma \ref{lem:SMvsMV0}.
\end{proof}

\begin{example} \label{ex1} Consider the generic polynomial system $\bff = (f_1,f_2, f_3) $ with
$$\begin{array}{rcl}
f_1 &=&  c_{11} x_1  +  c_{12} x_2 + c_{13} x_2^2 + c_{14} x_1^2 x_2 x_3  + c_{15} x_3^7\\
f_2 &=& c_{21}  x_1^2 + c_{22} x_1^3 + c_{23} x_1^2 x_2 + c_{24} x_3^3 + c_{25} x_2^7\\
f_3 &=& c_{31} x_1  + c_{32}  x_1x_2 + c_{33} x_3^2+ c_{34} x_2 x_3^3  + c_{35} x_2^7\\
\end{array}$$
with support family $\cA = (\cA_1,\cA_2, \cA_3)$, where
$$\begin{array}{rcl}
\cA_1 & =  &\{ (1,0,0), (0,1,0), (0,2,0), (2,1,1), (0,0,7)\} \\
\cA_2 & =  &\{ (2,0,0), (3,0,0), (2,1,0), (0,0,3), (0,7,0)\} \\
\cA_3 & =  &\{ (1,0,0), (1,1,0), (0,0,2), (0,1,3), (0,7,0)\}
\end{array}$$
satisfying assumptions (H1) and (H3). Then, Proposition \ref{prop:mult0} states that $0$ is an isolated common root of $\bff$ with multiplicity
$$\mult_0(\bff) = MV_3(\cA^0) - MV_3(\cA) = 147-144 =3.$$
\end{example}

\medskip

In order to restate the formula in the previous proposition by means of a mixed integral of suitable convex functions, we first introduce further notation and prove some auxiliary results.

For $1\le j \le n$, let $Q_j=\conv(\cA_j)$ and $\Delta_j=\conv\{0, \lambda_{1j}e_1,
\dots, \lambda_{nj}e_n\}$, where
\begin{equation}\label{eq:lambdaij}
\lambda_{ij}=\mbox{min}\{\mu\in \N \mid \mu e_i \in Q_j\} \quad \hbox { for } 1\le i \le n.
\end{equation}

Let $\pi:\R^n \rightarrow \R^{n-1}$ be the projection to the first $n-1$ coordinates.
As in Section \ref{section:general}, let
$\sigma_j:\pi(Q_j) \rightarrow \R$ denote the concave function that
parameterizes the upper envelope of $Q_j$ and $\rho_j:\pi(Q_j)
\rightarrow \R$ the convex function that parameterizes its
lower envelope. Since $\pi(\Delta_j) \subset \pi(Q_j)$, we may consider
\begin{equation} \label{eq:restricciones}
\overline{\sigma}_j=
\sigma_j|_{\pi(\Delta_j)} \quad \hbox{ and } \quad \overline{\rho}_j=
\rho_j|_{\pi(\Delta_j)},
\end{equation}
the restrictions of these functions to $\pi(\Delta_j)$.

For a non-empty set $J \subset \{1,\dots, n\}$, we denote
$$\Delta_J := \sum_{j\in J} \Delta_j, \qquad Q_J := \sum_{j\in J} Q_j .$$

\begin{lemma} \label{lem:normal<0} Let $J \subset \{1, \dots, n\}$ be a non-empty set. Then,
every facet of $\Delta_J$ that is not contained in a hyperplane $\{x_i=0\}$, for $1\le i \le n$, has an inner normal vector with all negative coordinates.
We will call these facets the \emph{non-trivial} facets of $\Delta_J$.
\end{lemma}

\begin{proof} If $J=\{j\}$ for some $1 \le j \le n$, the result is straightforward
because the only facet satisfying the required conditions is
$F=\conv\{\lambda_{1j}e_1, \dots, \lambda_{nj}e_n\}$, and $\lambda_{ij} \in \N$ for every $1\le i \le n$.

Let $F$ be a non-trivial facet of $\Delta_J$ and $\eta=(\eta_1, \dots, \eta_n)$ an inner normal vector of $F$.  Then, $F=\sum_{j\in J}F_j$, where $F_j$ is a face of $\Delta_j$ with
inner normal vector $\eta$. For every $1 \le i \le n$, since $F$ is not contained in the hyperplane $\{x_i=0\}$,
there exists $j_i\in J$ such that $\lambda_{ij_i}e_i\in F_{j_i}$;
then,
\begin{equation}\label{eq:eta}
0= \eta.\, 0 \ge \eta.\, \lambda_{ij_i}e_i =\eta_i
\lambda_{ij_i}
\end{equation}
and, so $\eta_i\le 0$.

If $0\in F_{j}$ for some $j\in J$, then $\eta.\, q \ge \eta\cdot
0=0$ for every $q \in \Delta_{j}$; in particular, $\eta_k
\lambda_{kj} = \eta.\, \lambda_{kj} e_k \ge \eta.\, 0 =0$ for
every $1\le k \le n$. This implies that $\eta = 0$, a
contradiction. Then, $0\notin F_j$ for every $j\in J$, the
inequalities in (\ref{eq:eta}) are strict and, therefore, $\eta_i
<0$ for every $1\le i \le n$.
\end{proof}

\begin{lemma} \label{lem:cero en borde} Let $J \subset \{1, \dots, n\}$ be a non-empty set.
 Then, for every point $x$ in a
non-trivial facet of $\pi(\Delta_J)$ we have that
$(\boxplus'_{j\in J}\rho_j)(x)=0$.
\end{lemma}

\begin{proof}
If $J=\{j\}$, we have $x \in \conv\{\lambda_{1j}\pi(e_{1}), \dots,\lambda_{n-1,j}\pi(e_{n-1}) \}$, that is, $x = \sum_{i=1}^{n-1}t_i \lambda_{ij}\pi(e_{i})$ for $t_i\ge 0$ with $\sum_{i=1}^{n-1}t_i = 1$. Then, since $\rho_j$ is convex, $0 \le \rho_j(x)\le \sum_{i=1}^{n-1}t_i \rho_j(\lambda_{ij}\pi(e_{i}))=0.$

\smallskip
If $\#J >1$, let $x$ be in a nontrivial facet $F$ of $\pi(\Delta_J)$. We have that $F= \sum_{j\in J} F_j$, with  $F_j$ a face of $\pi(\Delta_j)$ such that $0\notin F_j$; then,  $x=\sum_{j\in J}p_j$ with $p_j\in F_j$.  Hence, $\rho_j(p_j)=0$ and, by the definition of
$\boxplus'_{j\in J}\rho_j$, it follows that $0\le (\boxplus'_{j\in J}\rho_j)(x)\le \sum_{j\in J}\rho_j(p_j)=0.$
\end{proof}

\begin{lemma} \label{lem:suma pisos} For every non-empty
subset $J$ of $\{1, \dots, n\}$, the convex function
$\boxplus'_{j\in J}\overline{\rho}_j$
defined over $\pi(\Delta_J)$ parameterizes the lower envelope of $Q_J$ over the points of $\pi(\Delta_J)$.
\end{lemma}

\begin{proof} For every $J\subset \{1, \dots, n\}$, we denote
$$P_J:= \sum_{j\in J} \pi(Q_j), \quad D_J:=\sum_{j\in J}\pi(\Delta_j)$$
and
$\overline{\boxplus'_{j\in J}\rho_j} $ to the restriction of $\boxplus'_{j\in J}\rho_j: P_J \to \R$ to $D_J \subset P_J$. With this notation, we have to prove that
\begin{equation}\label{eq:boxplus}
\boxplus'_{j\in J}\overline{\rho}_j =\overline{\boxplus'_{j\in J}\rho_j}.
\end{equation}

Before proceeding, we will state three basic results that will be applied throughout the proof.
We use the notation $$\rho_J := \boxplus'_{j\in J}\rho_j.$$

\medskip

\noindent \textsc{Claim I.}  If $p_1$ lies in a non-trivial facet of $D_J$ and $p_2 \in P_J$, then for every $x$ lying on the line segment $p_1 p_2$, we have $\rho_J(x) \le \rho_J (p_2)$: as $x = (1-t)p_1 + t p_2$ for $0\le t \le 1$, $\rho_J$ is convex and $\rho_J\equiv 0$ on the non-trivial facets of $D_J$, $\rho_J(x) \le (1-t)\rho_J(p_1) + t \,\rho_J(p_2) = t \, \rho_J(p_2)$.

\smallskip
\noindent \textsc{Claim II.} If $p_1 \in D_J$ and $p_2 \notin D_J$, then for every $x\ne p_2$ lying on the line segment $p_1 p_2$, since $D_J$ is a convex set, $d(p_1, D_J) < d(p_2, D_J)$, where $d(\cdot, D_J)$ is the distance to $D_J$.

\smallskip
\noindent \textsc{Claim III.} If $p_1 \in D_J$ and $p_2\in (\R_{\ge0})^{n-1} \setminus D_J$ there exists $t\in (0, 1]$ such that $tp_1 +(1-t) p_2 $ lies in a non-trivial facet of $D_J$.

\medskip

The proof will be done recursively.  For a fixed non-empty set $J\subset \{1, \dots, n\}$,   let $J_1, J_2$ be disjoints sets such that $J = J_1\cup
J_2$ and assume that identity (\ref{eq:boxplus}) holds for each of them. We will prove that if $\overline{\rho}_{J_k}:=\boxplus'_{j\in J_k}\overline{\rho}_j$, for $k=1,2$, then
$\overline{\rho}_{J_1}\boxplus'\overline{\rho}_{J_2}=
\overline{\rho_{J_1}\boxplus'\rho_{J_2}}$.

\medskip

Let $x \in D_J$. Then, there exist $y_0\in
D_{J_1}$  and $z_0 \in D_{J_2}$ such that $x=y_0+z_0$. Let $y'\in
P_{J_1}$  and $z' \in P_{J_2}$ be such that $x = y' + z'$ and
$\overline{\rho_{J_1}\boxplus' {\rho_{J_2}}} (x) = \rho_{J_1}(y') +\rho_{J_2}(z')$.
If $y'\in
D_{J_1}$  and $z' \in D_{J_2}$ the result follows.

We first show that there exist $y'$ and $z'$ as before satisfying that  $y' \in D_{J_1}$ or $z' \in D_{J_2}$. For every $0 \le t \le  1$,  if $y_t = (1-t)y_0+ ty'$ and $z_t=(1-t)z_0 + tz'$,  then $x = y_t + z_t$. If $y' \notin D_{J_1}$ and $z' \notin D_{J_2}$,
 there exist $ 0 < t_1, t_2 \le 1$ such that
$y_{t_1} $ and
$z_{t_2}$ lie in  non-trivial facets of $D_{J_1}$ and $D_{J_2}$ respectively.
Consider $t_0=\min\{t_1,t_2\}$; then
$x = y_{t_0}+ z_{t_0}$ and, by Claim I, $\overline{\rho_{J_1}\boxplus {\rho_{J_2}}} (x) = \rho_{J_1}(y_{t_0}) +\rho_{J_2}(z_{t_0})$.

Now, without loss of generality, assume that $z' \in D_{J_2}$.  Consider the compact set
$$C_x = \{ y \in P_{J_1} \mid  x-y \in D_{J_2}
\mbox{ and }  \overline{\rho_{J_1}\boxplus' {\rho_{J_2}}} (x)  = \rho_{J_1}(y) + {\rho_{J_2}}(x-y)\}.$$
We will prove that $C_x\cap D_{J_1} \ne \emptyset$. If not, let $y \in C_x$ be such that
$d(C_x, D_{J_1}) = d(y, D_{J_1}) > 0$.

First, assume that $z:=x-y $ does not lie in a non-trivial facet of $D_{J_2}$. This implies that $z + w \in D_{J_2}$ for every $w$ with sufficiently small non-negative coordinates.
Let $0<\epsilon <1$ such that $(1-\epsilon)y \notin D_{J_1}$ and that $z+\epsilon y \in D_{J_2}$. Claims III and I imply that $\rho_{J_1}((1-\epsilon)y) \le \rho_{J_1}(y)$ and that
$\rho_{J_2}(z +\epsilon y) \le \rho_{J_2}(z)$ and, therefore, $\rho_{J_1}\boxplus' {\rho_{J_2}}(x) = \rho_{J_1}((1-\epsilon)y) + \rho_{J_2}(z +\epsilon y)$. As, by Claim II,
$ d((1-\epsilon)y,D_{J_1})<d(y, D_{J_1}) $ we have a contradiction.

Assume now that $z:=x-y $ lies in non-trivial facets of $D_{J_2}$.

Recall that $x = y_0 + z_0$ with $y_0 \in D_{J_1}$, $z_0 \in D_{J_2}$. If $z$ and $z_0$ lie in the same non-trivial facet of $D_{J_2}$, then the line segment $zz_0$ is contained in this facet. On the other hand, there exists $0\le t \le 1$ such that $(1-t)y_0+ ty$ lies in a non-trivial facet of  $D_{J_1}$. Therefore, $x= ((1-t)y_0+ ty) + ((1-t)z_0 + tz)$,   $\rho_{J_1}\boxplus' {\rho_{J_2}} (x)  = \rho_{J_1}((1-t)y_0+ ty) + {\rho_{J_2}}((1-t)z_0 + tz) = 0 $ and so, $(1-t)y_0+ ty \in C_x\cap D_{J_1} $, which is a contradiction.

If $z_0$ does not lie in any of the non-trivial facets of
$D_{J_2}$ containing $z$, let  $\eta^1, \dots, \eta^k$ be inner
normal vectors to these facets and consider the hyperplanes
parallel to them and  containing $y$, which are defined by the
equations  $\eta^\ell.\, (Y-y) = 0$ for $1 \le \ell \le k$.  As
$\eta^\ell.\, y + \eta^\ell.\, z = \eta^\ell.\, y_0 +
\eta^\ell.\, z_0$ and $\eta^\ell.\, z < \eta^\ell.\, z_0$,
then $\eta^\ell.\, y_0 < \eta^\ell.\, y$. In addition, since
all the coordinates of $\eta^\ell$ are negative (see Lemma
\ref{lem:normal<0}) and $y \in (\R_{\ge 0})^{n-1}$, then
$\eta^\ell.\, y<0$. Therefore, the hyperplane  $\eta^\ell\cdot
(Y-y) = 0$ intersects the line segment $0y_0$ in a point
$\lambda_\ell y_0$ with $0\le \lambda_\ell \le 1$. If $\lambda=
\max\{\lambda_\ell \ / \ 1 \le \ell \le k\}$, consider $y_t=
(1-t)y+ t\lambda y_0$ and $z_t = x -y_t$ for $0 \le t \le 1$. For
$t$ sufficiently small, we will show that $z_t \in D_{J_2}$, that
$\rho_{J_1}\boxplus' {\rho_{J_2}} (x)  = \rho_{J_1}(y_t) +
{\rho_{J_2}}(z_t)$ and that $d(y_t, D_{J_1}) < d(y, D_{J_1})$,
which leads to a contradiction.

For $1 \le \ell \le k$, as $\lambda\ge \lambda_\ell$,
$\eta^\ell.\, (y-\lambda y_0) \ge 0$; then  $\eta^\ell\cdot
(y-y_t) \ge 0$ and so $\eta^\ell.\, z_t = \eta^\ell.\, z +
\eta^\ell.\, (y - y_t) \ge \eta^\ell.\, z$. If $z$ lies in a
trivial facet of $D_{J_2}$, that is, $z_i =0$ for some $1\le i \le
n$, then $y_i = x_i$; as $(y_{0} )_i \le x_i$, we have that
$(z_t)_i = t (y_i - \lambda (y_{0})_i) \ge 0$. Taking $t$
sufficiently small, $z_t$ satisfies all the remaining inequalities
defining $D_{J_2}$ and so, $z_t \in D_{J_2}$. Moreover, since
$y\notin D_{J_1}$, for $t$ sufficiently small, $y_t \notin
D_{J_1}$. Then, by Claim I, $\rho_{J_1}(y_t) \le \rho_{J_1}(y)$.
On the other hand,  $z_t$ lies in the same non-trivial facet of
$D_{J_2}$ as $z$, namely, the facet defined by $\eta^{\ell_0}
\cdot(Z- z) = 0$ for $\ell_0$ such that $\lambda =
\lambda_{\ell_0}$ and, therefore, $\rho_{J_2}(z_t) = 0$. We
conclude that $ \rho_{J_1}(y_t) + {\rho_{J_2}}(z_t) =
\rho_{J_1}\boxplus' {\rho_{J_2}} (x)$. Finally, the inequality
$d(y_t, D_{J_1})< d(y, D_{J_1})$ holds by Claim II.
\end{proof}

For every $1\le j \le n$, let $Q_j^0=\conv(\cA_j\cup\{0\})$ and
$\sigma_j^0, \rho_j^0$ the functions that parameterize its upper
and lower envelopes respectively. Assumption (H3) ensures that
$\pi(Q_j^0) = \pi(Q_j)$.

\begin{lemma} \label{lem:piso ext} For every $1\le j \le n$, $\rho_j^0(x)= \begin{cases} 0  & \mbox{if } x\in \pi(\Delta_j) \\ \rho_j(x) & \mbox{if }
x\not\in \pi(\Delta_j) \end{cases}$  and $\sigma_j^0 =\sigma_j$.
\end{lemma}

\begin{proof}
Since $Q_j \subset Q_j^0$, then $\rho_j^0(x) \le \rho_j (x)$ and $\sigma_j(x) \le
\sigma_j^0(x)$ for every $x\in \pi(Q_j^0)$.

If $x\in \pi(\Delta_j)$, there exists $x_n \ge 0$ such that $(x,x_n)\in \Delta_j$. Then,
$(x, x_n)=\sum_{i=1}^nt_i\lambda_{ij}e_i$, where
$\sum_{i=1}^nt_i=1$ and $t_i \ge 0$ for every $ 1 \le i \le n$.
Taking $y=\sum_{i=1}^{n-1}t_i\lambda_{ij}e_i$, we have that $\pi(y)=x$,
$(y)_n=0$ and $y \in Q_j^0$. Hence, $\rho_j^0(x)=0.$

Consider now $x\in \pi(Q_j)\backslash \pi(\Delta_j)$. Take
$(x,\rho_j^0(x)) \in Q_j^0 = \mbox{conv}(Q_j \cup \{ 0 \})$. Then,
$(x,\rho_j^0(x))=t q$, with $q\in Q_j$ and $0< t \le 1$. Since $(x,\rho_j^0(x))\notin \Delta_j$,  there exists $0<t'<1$ such that $t'(x,\rho_j^0(x))$ lies in the nontrivial facet of $\Delta_j$ and so, $q':= t'(x,\rho_j^0(x))\in Q_j$. Then, the line segment $q q'$ is contained in $Q_j$;
in particular, $(x,\rho_j^0(x))\in Q_j$. It follows that $\rho_j(x)\le
\rho_j^0(x)$.

If $x\in \pi(Q_j)=\pi(Q_j^0)\subset \R^{n-1}$, consider $(x, \sigma_j^0(x))\in Q_j^0$. Then, $(x,\sigma_j^0(x)) =  t q$ with $q\in Q_j$ and $0\le t \le 1$. If $y=tq +(1-t)\lambda_{nj}e_n\in Q_j$, then $\pi(y)=x$ and so, $\sigma_j(x)\ge (y)_n= \sigma_j^0(x) + (1-t)\lambda_{nj} \ge
\sigma_j^0(x)$.
\end{proof}

Now, we can restate the formula for the multiplicity of the origin in Proposition \ref{prop:mult0} as a mixed integral of convex functions:

\begin{theorem}\label{teo:MIcasoTocandoEjes} Let $\cA = (\cA_1, \dots, \cA_n)$ be a family of finite sets in
$(\Z_{\ge 0})^n$ satisfying assumptions (H1) and (H3). Let $\bff = (f_1, \dots, f_n)$ be a generic system of sparse polynomials in $\C[x_1, \dots, x_n]$ supported on  $\cA$. For every $1\le j \le n$, let $\bar \rho_j$ be the convex function defined in (\ref{eq:restricciones}). Then, the origin is an isolated
common zero of $\bff$  and $$\mult_0 (\bff) = MI_n'(\overline{\rho}_1, \dots,
\overline{\rho}_n).$$
\end{theorem}

\begin{proof} By Proposition \ref{prop:mult0}, it suffices to show that
$$MV_n(\cA^0)-MV_n(\cA) = MI_n'(\overline{\rho}_1, \dots,
\overline{\rho}_n).$$

For every $1 \le j \le n$, consider $\nu_j \in \R$
such that $\nu_j \ge \max(\rho_j)\ge \max(\overline{\rho}_j)$.

For $J\subset \{1,\dots, n\}$, let $D_J =\sum_{j\in J}\pi(\Delta_j)$ and $\nu_J= \sum_{j \in J}\nu_j$. Since $\nu_J \ge \max (\boxplus'_{j\in J}\overline{\rho}_j)$, we have that
$$\int_{D_J}\boxplus'_{j\in J}\overline{\rho}_j\ dx_1\dots dx_n = \nu_J Vol_{n-1}(D_J)-
Vol_n\left((D_J)_{\boxplus_{j\in J}'\overline{\rho}_j,\nu_J}\right).$$
Then, by Lemma \ref{lem:suma pisos},
$$ \int_{D_J}\boxplus'_{j\in J}\overline{\rho}_j\, dx_1\dots dx_n
= Vol_n\left( (D_J)_{0, \nu_J}\right)- Vol_n\left((D_J)_{(\boxplus_{j\in J}'\rho_j)|_{D_J},\nu_J}\right).$$
Now, if $P_J= \sum_{j\in J}\pi(Q_j^0) =\sum_{j\in J}\pi(Q_j)$, Lemma \ref{lem:piso ext} implies that
$$Vol_n\left( (D_J)_{0, \nu_J}\right)- Vol_n\left((D_J)_{(\boxplus_{j\in J}'\rho_j)|_{D_J},\nu_J}\right)={}$$ $${}= Vol_n( (P_J)_{\boxplus_{j\in J}'\rho_j^0,\nu_J}) - Vol_n((P_J)_{\boxplus_{j\in J}'\rho_j,\nu_J}) ={} $$
$$ {}= Vol_n\left((P_J)_{\boxplus_{j\in J}'\rho_j^0,\boxplus_{j\in J}\sigma_j^0}\right) -
Vol_n\left((P_J)_{\boxplus_{j\in J}'\rho_j,\boxplus_{j\in J}\sigma_j}\right).$$
Finally, by Remark \ref{rem:def x piso y  techo},
$$ Vol_n\left((P_J)_{\boxplus_{j\in J}'\rho_j^0,\boxplus_{j\in J}\sigma_j^0}\right)=
Vol_n\Big(\sum_{j\in J} Q_j^0\Big)$$  and $$ \ Vol_n\left((P_J)_{\boxplus_{j\in J}'\rho_j,\boxplus_{j\in J}\sigma_j}\right)= Vol_n\Big(\sum_{j\in J}Q_j\Big),$$
and so,
$$\int_{D_J}\boxplus'_{j\in J}\overline{\rho}_j\, dx_1\dots dx_n
=Vol_n\Big(\sum_{j\in J} Q_j^0\Big) - Vol_n\Big(\sum_{j\in J}Q_j\Big).$$

The theorem follows from the definitions of the mixed integral and the mixed volume.
\end{proof}

\begin{example} \label{ex2} Consider the generic sparse polynomial system
$$\begin{array}{l}
f_1 = c_{1,20} x_1^2 +c_{1,11} x_1 x_2+c_{1,04}x_2^4 +c_{1,13} x_1 x_2^3+c_{1,33} x_1^3 x_2^3 \\
f_2 = c_{2,40} x_1^4 +c_{2,21} x_1^2  x_2+c_{2,04}x_2^4 +c_{2,25} x_1^2 x_2^5+c_{2,13} x_1 x_2^3
\end{array}$$
with supports
$$
\cA_1 =\{(2,0), (1,1),(0,4), (1,3),(3,3)\}, \qquad
\cA_2 =\{(4,0), (2,1),(0,4), (2,5),(1,3)\}.
$$
Here, we have $\Delta_1 = \conv\{ (0,0), (2,0), (0,4)\}$ and $\Delta_2= \conv\{(0,0), (4,0), (0,4)\}$.

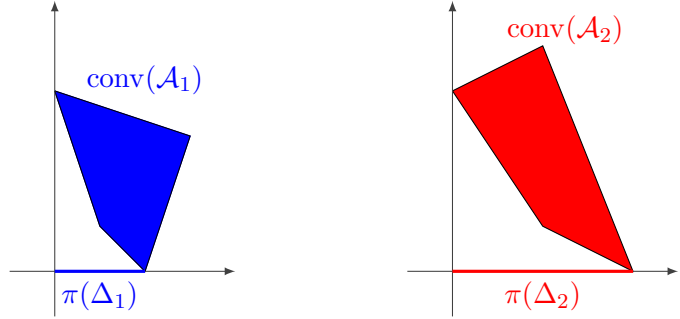
\begin{center}
\begin{tikzpicture}[scale=0.6]
\draw[fill=blue, color=blue] (0,4)--(1,1)--(2,0)--(3,3)--(0,4);
\draw[color=black] (0,4)--(1,1)--(2,0)--(3,3)--(0,4);
\draw[-latex,color=darkgray,thin] (-1,0) -- (4,0);
\draw[-latex,color=darkgray,thin] (0,-1) -- (0,6);
\draw[color=blue, very thick] (0,0)--(2,0);
\draw[shift={(2,3.7)}] node[above] {\textcolor{blue}{$\conv(\cA_1)$}};
\draw[shift={(1,0)}] node[below] {\textcolor{blue}{$\pi(\Delta_1)$}};
\end{tikzpicture} \hspace{2cm}
\begin{tikzpicture}[scale=0.6]
\draw[color=black, fill=red] (0,4)--(2,1)--(4,0)--(2,5)--(0,4);
\draw[-latex,color=darkgray,thin] (-1,0) -- (5,0);
\draw[-latex,color=darkgray,thin] (0,-1) -- (0,6);
\draw[color=red, very thick] (0,0)--(4,0);
\draw[shift={(2.5,4.8)}] node[above] {\textcolor{red}{$\conv(\cA_2)$}};
\draw[shift={(2,0)}] node[below] {\textcolor{red}{$\pi(\Delta_2)$}};
\end{tikzpicture}
\end{center}

To compute the multiplicity of the origin following Theorem \ref{teo:MIcasoTocandoEjes}, consider the convex functions $\overline \rho_1 : \pi(\Delta_1) \to \R$ and $\overline \rho_2 : \pi(\Delta_2) \to \R$:

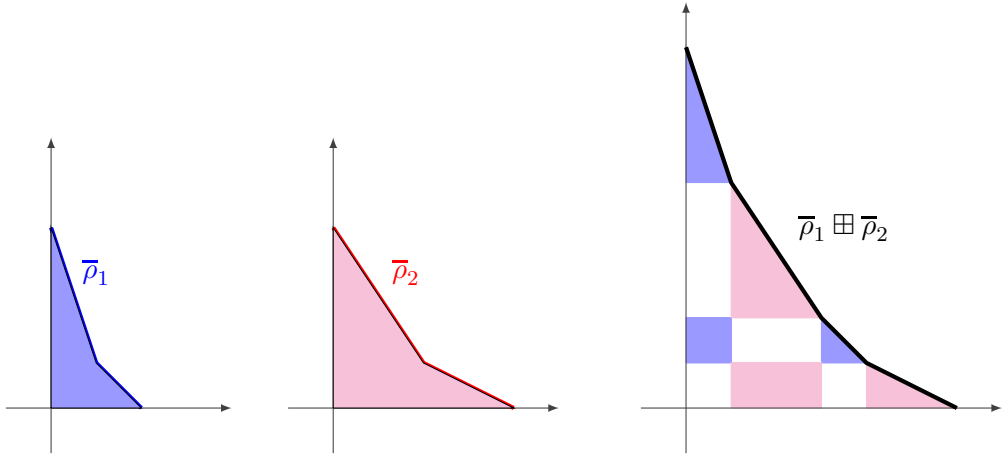
\begin{center}
\begin{tikzpicture}[scale=0.6]
\draw[-latex,color=darkgray,thin] (-1,0) -- (4,0);
\draw[-latex,color=darkgray,thin] (0,-1) -- (0,6);
\draw[color=blue, ultra thick] (0,4)--(1,1)--(2,0);
\draw[fill=blue!40] (0,0)--(0,4)--(1,1)--(2,0)--(0,0);
\draw[shift={(1,2.5)}] node[above] {\textcolor{blue}{{$\overline{\rho}_1$}}};
\end{tikzpicture}\hspace{6mm}
\begin{tikzpicture}[scale=0.6]
\draw[-latex,color=darkgray,thin] (-1,0) -- (5,0);
\draw[-latex,color=darkgray,thin] (0,-1) -- (0,6);
\draw[shift={(0,0)},color=red, ultra thick] (0,4)--(2,1)--(4,0);
\draw[fill=magenta!30] (0,0)--(0,4)--(2,1)--(4,0)--(0,0);
\draw[shift={(1.5,2.5)}] node[above] {\textcolor{red}{{ $\overline{\rho}_2$}}};
\end{tikzpicture} \hspace{8mm}
\begin{tikzpicture}[scale=0.6]
\draw[shift={(0,5)},fill=blue!40, color=blue!40] (0,0)--(1,0)--(0,3)--(0,0);
\draw[shift={(0,1)},fill=blue!40, color=blue!40] (0,0)--(1,0)--(1,1)--(0,1)--(0,0);
\draw[shift={(3,1)},fill=blue!40, color=blue!40] (0,0)--(1,0)--(0,1)--(0,0);
\draw[shift={(4,0)},fill=magenta!30, color=magenta!30] (0,0)--(2,0)--(0,1)--(0,0);
\draw[shift={(1,0)},fill=magenta!30, color=magenta!30] (0,0)--(2,0)--(2,1)--(0,1)--(0,0);
\draw[shift={(1,2)},fill=magenta!30, color=magenta!30] (0,0)--(2,0)--(0,3)--(0,0);
\draw[-latex,color=darkgray,thin] (-1,0) -- (7,0);
\draw[-latex,color=darkgray,thin] (0,-1) -- (0,9);
\draw[shift={(0,0)},color=black, ultra thick] (0,8)--(1,5)--(3,2)--(4,1)--(6,0);
\draw[shift={(3.5,4)}] node {{$\overline{\rho}_1\boxplus \overline{\rho}_2$}};
\end{tikzpicture}
\end{center}
Therefore,
$$\mult_0(\bff) = \displaystyle MI'_2(\overline{\rho}_1, \overline{\rho}_2) = \int_0^6 \overline{\rho}_1\boxplus \overline{\rho}_2 (x)\, dx - \int_0^2 \overline{\rho}_1(x)\,  dx - \int_0^4 \overline{\rho}_2 (x) \, dx=  7.$$

\end{example}

\begin{remark} The computation of the multiplicity of the origin by means of mixed integrals following Theorem \ref{teo:MIcasoTocandoEjes} may involve smaller polytopes than its computation using mixed volumes according to Proposition \ref{prop:mult0}, since it depends only on the points of the lower envelopes of the polytopes $Q_j=\conv(\cA_j)$  that lie above the simplices $\pi(\Delta_j)$ for $j=1,\dots, n$.

Following \cite{HS97}, this computation can also be done by locating the stable mixed cells with positive inner normals in a subdivision of $\cA^0$ induced by a suitable lifting, and computing and adding the mixed volumes of those cells, which may also involve smaller polytopes. Moreover, the proof of \cite[Theorem 2]{HS97} implies that the mixed integral in Theorem \ref{teo:MIcasoTocandoEjes} also counts the number of Puiseux series expansions around the origin of the solution set of the system under generic perturbation of the constant terms of the polynomials.
\end{remark}

\subsection{General case} \label{subsec:generalcase}

Consider now a family $\cA=(\cA_1, \dots, \cA_n)$ of finite sets
in $(\Z_{\ge 0})^n$ satisfying conditions (H1) and (H2).
 Let $\bff=(f_1, \dots, f_n)$  be a system of generic sparse polynomials in $\C[x_1, \dots,
x_n]$ supported on $\cA$.

For $M\in \Z_{> 0}$, let $\Delta_M:= \{ M e_i\}_{i=1}^n$ and,  for all $ 1 \le j \le n$, let
$\cA_{j}^{\Delta_M}:=\cA_j\cup \Delta_M$ and $\cA_{j}^{\Delta_M,0}:=\cA_{j}^{\Delta_M}\cup \{ 0\}$. Set $\cA^{\Delta_M}:= (\cA_1^{\Delta_M},\dots,\cA_n^{\Delta_M})$ and $\cA^{\Delta_M,0}:=(\cA_1^{\Delta_M,0},\dots,\cA_n^{\Delta_M,0})$.

\begin{proposition} \label{prop:Mgrande} With the previous assumptions and notation, we have that $0$ is an isolated common zero of $\bff$ and, for every $M\gg 0$, its multiplicity is
$$\mult_0(\bff)= MV_n(\cA^{\Delta_M,0}) -
MV_n(\cA^{\Delta_M}).$$
Moreover, the identity holds for every $M> \mult_0(\bff)$. In particular, it suffices to take $M= MV_n(\cA^0) - MV_n(\cA) +1$.
\end{proposition}

\begin{proof} Conditions (H1) and (H2) imply that $0$ is an isolated common zero of the generic system $\bff$ supported on $\cA$.

Take  $M> \mult_0(\bff)$  and consider polynomials
$$g_j=f_j+\sum_{i=1}^n c_{j,Me_i}\, x_i^{M}$$ with support sets
$\cA_{j}^{\Delta_M}=\cA_j\cup \Delta_M$  and generic coefficients for all $ 1 \le j \le n$.

Since $\cA^{\Delta_M}= (\cA_1^{\Delta_M},\dots,\cA_n^{\Delta_M}) $ fulfills the conditions (H1) and (H3) stated in Section \ref{section:toca ejes}, by Proposition \ref{prop:mult0} the multiplicity of the origin as a common isolated root of
$\bfg:= (g_1, \dots, g_n)$ is $\mult_0(\bfg)=MV_n(\cA^{\Delta_M,0}) -
MV_n(\cA^{\Delta_M}).$

Let us prove that $\mult_0(\bff) = \mult_0(\bfg)$. To do so, we consider the matrices $S_k(\bff,0)$ and $S_k(\bfg,0)$, for $k\ge 0$, introduced in Section \ref{sec:DZ}. Note that, since $M>\mult_0(\bff)$, in order to compute $\mult_0(\bff)$, it suffices to compare the dimensions of the nullspaces of the  matrices $S_k(\bff,0)$ for $0\le k \le M-1$. Now, for every $k\le M-1$, $\alpha, \beta \in (\Z_{\ge 0})^n$, with $|\alpha| \le k$ and $|\beta|\le k-1$, and every $1\le j \le n$, we have that
$$ (S_k(\bff,0))_{(\beta, j) , \alpha} = \dfrac{1}{\alpha!}\dfrac{\partial^{|\alpha|}}{\partial x^\alpha} (x^\beta f_j) (0) = \dfrac{1}{\alpha!}\dfrac{\partial^{|\alpha|}}{\partial x^\alpha} (x^\beta g_j) (0)=  (S_k(\bfg,0))_{(\beta, j) , \alpha}.$$
Since the dimensions of the nullspaces of $S_k(\bff,0)= S_k(\bfg,0)$ stabilize for $k<M$, then, $\mult_0(\bff)= \mult_0(\bfg)$.

The fact that we can take $M= MV_n(\cA^0) - MV_n(\cA) +1$ follows from inequality (\ref{eq:cotamult}).
\end{proof}

From the previous result and Theorem \ref{teo:MIcasoTocandoEjes} we can
express the multiplicity of the origin as an isolated zero of a generic sparse
system via mixed integrals:

\begin{corollary} \label{cor:mult0MI} Let $\cA =(\cA_1, \dots, \cA_n)$ be a family of finite sets in
$(\Z_{\ge 0})^n$ satisfying assumptions (H1) and (H2).   Let $\bff= (f_1, \dots, f_n) $  be
a generic family of polynomials in $\C[x_1, \dots, x_n]$ supported on $\cA$. Let
$M:=MV_n(\cA^0)- MV_n(\cA)+1$ and, for $1\le j \le n$, let $\rho_{j}^{\Delta_M}$ be the
convex function that parameterizes the lower envelope of the
polytope $\conv(\cA_j^{\Delta_M})$ and $\overline{\rho_{j}}^{\Delta_M}$ its restriction defined as in (\ref{eq:restricciones}). Then, $$\mult_0(\bff) = MI_n'(\overline{\rho_1}^{\Delta_M}, \dots,
\overline{\rho_n}^{\Delta_M}).$$
\end{corollary}

The following property enables us to deal with smaller support sets when computing multiplicities.

\begin{proposition} \label{prop:sacar monomios}
Let $\bff = (f_1, \dots, f_n)$ be a generic system of polynomials in
$\C[x_1, \dots, x_n]$ supported on a family $\cA= (\cA_1,
\dots, \cA_n)$ of finite subsets of $(\Z_{\ge 0})^n$. Assume that $0$ is an isolated
common zero of $\bff$. Let $f_1 = \sum_{a\in\cA_1}c_{1,a} x^a$.
If $\alpha, \alpha + \beta \in \cA_1$ with $\beta \in (\Z_{\ge 0})^n\setminus\{ 0 \}$, then
$$\mult_0(\bff) = \mult_0(f_1 -c_{1,\alpha+\beta}x^{\alpha+\beta}, f_2,\dots, f_n).$$
\end{proposition}

\begin{proof}
Let $h_1, \dots, h_n$ be polynomials of the form $h_j = f_j
+ \sum_{i=1}^n c_{j, M e_i}\, x_i^{M}$ with $c_{j,M e_i}\in \C$
generic coefficients and $M\in \N$ sufficiently big such that
$$\mbox{mult}_0(f_1, \dots, f_n) = \mbox{mult}_0(h_1, \dots,
h_n),$$
$$ \mbox{mult}_0(f_1- c_{1,\alpha+\beta}x^{\alpha+\beta}, f_2,
\dots, f_n) = \mbox{mult}_0(h_1- c_{1,\alpha+\beta}x^{\alpha+\beta},h_2,
\dots, h_n),$$ $\alpha+\beta\ne M e_i$ for all $1\le i\le n$ and $\cA_1
\subset \mbox{conv}(\{0, M e_1, \dots, M e_n\})$.  The existence of $M$ is ensured by Proposition \ref{prop:Mgrande} and its proof.

To prove that
$ \mbox{mult}_0(h_1, \dots, h_n) = \mbox{mult}_0(h_1- c_{1,\alpha+\beta}x^{\alpha+\beta}, h_2, \dots, h_n)$,
by Proposition \ref{prop:mult0}, it suffices to show that
$\conv(\cA_1\cup \{M {e}_i\}_{i=1}^{n}\backslash
\{\alpha+\beta\}) =  \conv(\cA_1\cup \{M {e}_i\}_{i=1}^{n})$.
This follows from the fact that
$\alpha + \beta \in \conv(\{\alpha, M e_1, \dots, M e_n\}) $, since
$$\alpha + \beta = \Big(1- \frac{|\beta|}{M-|\alpha|}\Big) \alpha + \sum_{i=1}^n \Big(\frac{\beta_i}{M}+
\frac{|\beta|\alpha_i}{(M-|\alpha|)M}\Big) M {e}_i$$ is a convex
linear combination of $\alpha, Me_1, \dots, Me_n$.
\end{proof}

As a consequence of Proposition \ref{prop:sacar monomios} we are able to obtain a refined formula for the multiplicity of the origin for generic polynomials supported on a family $\cA$ satisfying conditions (H1) and (H2), with no need of adding extra points to the supports whenever they intersect the coordinate axes.

\begin{proposition} \label{prop:refined} Let $\cA = (\cA_1,\dots, \cA_n)$ be a family of finite subsets of $(\Z_{\ge 0 })^n$ satisfying assumptions (H1) and (H2), and let $\bff =
(f_1,\dots, f_n)$ be a generic sparse polynomial system supported on $\cA$. Let $M\in \Z$, $M\ge MV_n(\cA^0) - MV_n(\cA)+1$. Then, $0$ is an isolated common zero of $\bff$ with multiplicity
$$\mult_0(\bff) = MV_n(\cA_1^{M, 0},\dots, \cA_n^{M,0})-MV_n(\cA_1^{M},\dots, \cA_n^{M}),$$
where, for every $1\le j \le n$, $\cA_j^{M} := \cA_j \cup \big\{ M e_i : 1\le i \le n, \, \cA_j \cap \{ \mu e_i \mid \mu \in \Z_{\ge 0}\} = \emptyset\big\}$ and $\cA_j^{M,0}:= \cA_j^{M} \cup \{ 0 \}$.
\end{proposition}

\begin{example} Consider the generic polynomial system $\bff = (f_1,f_2, f_3) $ with
$$\begin{array}{rcl}
f_1 &=&  c_{11} x_1  +  c_{12} x_2 + c_{13} x_2^2 + c_{14} x_1^2 x_2 x_3 \\
f_2 &=& c_{21}  x_1^2 + c_{22} x_1^3 + c_{23} x_1^2 x_2 + c_{24} x_3^3 \\
f_3 &=& c_{31} x_1  + c_{32}  x_1x_2 + c_{33} x_3^2+ c_{34} x_2 x_3^3  \\
\end{array}$$
with support family $\cA = (\cA_1,\cA_2, \cA_3)$, where
$$\begin{array}{rcl}
\cA_1 & =  &\{ (1,0,0), (0,1,0), (0,2,0), (2,1,1)\} \\
\cA_2 & =  &\{ (2,0,0), (3,0,0), (2,1,0), (0,0,3)\} \\
\cA_3 & =  &\{ (1,0,0), (1,1,0), (0,0,2), (0,1,3)\}
\end{array}$$
which satisfies assumptions (H1) and (H2). Then, $0$ is an isolated common root of $\bff$. In order to compute its multiplicity according to Proposition \ref{prop:refined},  let
$$M := MV_3(\cA^0) - MV_3(\cA)+1 = 28-22 +1=7,$$
and consider the modified support sets
$$\begin{array}{rcl}
\cA_1^{7} & =  &\{ (1,0,0), (0,1,0), (0,2,0), (2,1,1), (0,0,7)\} \\
\cA_2^{7} & =  &\{ (2,0,0), (3,0,0), (2,1,0), (0,0,3), (0,7,0)\} \\
\cA_3^{7} & =  &\{ (1,0,0), (1,1,0), (0,0,2), (0,1,3), (0,7,0)\}
\end{array},$$
which coincide with the supports of the polynomials in Example \ref{ex1}. Therefore,
$$\mult_0(\bff) = MV_3(\cA_1^{7,0},\cA_2^{7,0}, \cA_3^{7,0}) - MV_3(\cA_1^{7},\cA_2^{7}, \cA_3^{7}) = 3.$$
\end{example}

\section{Multiplicity of other roots with zero coordinates} \label{sec:multother}

Let $\cA=(\cA_1, \dots, \cA_n)$ be a family of finite sets in
$(\Z_{\ge 0})^n$ and $ \bff = (f_1, \dots, f_n) \subset \C[x_1, \dots, x_n]$ a generic family of polynomials with support set $\cA$.

For $I \subset \{1, \dots, n\}$, recall that
$$J_I= \{ j\in \{1,\dots, n\} \mid \exists a\in \cA_j : a_i = 0 \ \forall i\in I\}$$
 is the set of indices of the
polynomials in $\bff$ that do not vanish identically under the
specialization $x_i = 0$ for every $i\in I$. Also, for every $j\in J_I$, we denote
$$\cA_j^I = \{ a\in \cA_j \mid a_i = 0 \ \forall i\in I\}.$$

Following \cite[Section 3.2.1]{HJS13}, the system $\bff$ has isolated common zeros lying in
$O_I:=\{x \in \C^n \mid x_i=0 \mbox{ if and only if } i \in I\}$ if and only if
\begin{itemize}
\item[(A1)] $\#I +\#J_I=n$,
\item[(A2)] for every $\widetilde I \subset I$, $\# \widetilde I + \# J_{\widetilde I} \ge n$,
\item[(A3)] for every $J \subset J_I$, $\dim(\sum_{j\in J} \cA^I_j)\ge \# J$.
\end{itemize}

From now on, we will consider a non-empty set $I\subset\{1,\dots, n\}$ satisfying the conditions above and we will study the multiplicity of the isolated common zeros of $\bff$ lying in $O_I$.

\subsection{Multiplicity of affine isolated roots} \label{sec:multaffine}

The aim of this section is to compute multiplicities of the isolated zeros of $\bff$ in $O_I$ in terms of mixed volumes and mixed integrals associated to the system supports. The key result that allows us to do this shows that these multiplicities coincide with the multiplicity of the origin as an isolated root of an associated generic sparse system:

\begin{theorem}\label{teo:multF=multG} Let $\cA=(\cA_1, \dots, \cA_n)$ be a family of finite
sets in $(\Z_{\ge 0})^n$
and $\bff = (f_1, \dots, f_n)$ a generic  sparse system of polynomials in
$\C[x_1, \dots, x_n]$ supported on $\cA$. Assume that $\emptyset \ne I\subset \{1,
\dots, n\}$ satisfies conditions (A1), (A2) and (A3).
Let $\zeta \in \C^n$ be an isolated zero of $\bff$ with $\zeta\in O_I$. Then
$$\mult_{\zeta}(\bff)= \mult_0(\bfg),$$ for a system $\bfg:=(g_j)_{j\notin J_I}$  of generic polynomials
with supports $\cB_j^I:=\pi_I (\cA_j)$ for every $j\notin J_I$, where $\pi_I: \Z^n \to \Z^{\# I}$ is the projection onto the coordinates indexed by $I$.
\end{theorem}

For this statement to make sense, we  need the following:

\begin{lemma} Under the previous assumptions and notation, let $\cB^I= (\cB_j^I)_{j\notin J_I}$. Then, $0\in \C^{\# I}$ is an isolated zero of a generic polynomial system supported on $\cB^I$.
\end{lemma}

\begin{proof}
It suffices to show that $\cB^I$ satisfies conditions (H1) and (H2) stated at the beginning of Section \ref{sec:mult0} (see  \cite[Proposition 6]{HJS13}).

By the definition of $J_I$, it follows that $0\notin \pi_I(\cA_j) = \cB_j^I$ for every $j\notin J_I$.

In order to simplify notation, we will index the coordinates of $\Z^{\# I}$ by the corresponding elements of $I$.

To prove that condition (H2) holds, we must show that $\# \widetilde I + \# J_{\widetilde I}(\cB^I)\ge \# I$ for every $\widetilde I \subset I$, where
$J_{\widetilde I}(\cB^I)=\{j\notin J_I \mid \exists b\in \cB_j^I : b_i = 0 \ \forall i\in \widetilde I\}.$
Now, for every $\widetilde I \subset I$, we have that
$$J_{\widetilde I}(\cA) = J_I \cup \{j\notin J_I \mid \exists a\in \cA_j : a_i = 0 \ \forall i\in \widetilde I\} = J_I \cup J_{\widetilde I}(\cB^I).$$
Under assumption (A2) on $I$, the inequality $\#\widetilde I  +\# J_{\widetilde I}(\cA) \ge n$ holds; then,
$$\#\widetilde I  +\# J_{\widetilde I}(\cB^I) = \#\widetilde I+ \# J_{\widetilde I} (\cA)- \# J_I \ge n- \# J_I = \# I,$$
where the last identity follows from assumption (A1).
\end{proof}

In order to prove Theorem \ref{teo:multF=multG}, we first introduce some notation and prove some auxiliary results.

For a polynomial $g\in \C[x_1, \dots , x_n]$, $g_I$ will denote the polynomial in $\C[(x_i)_{i \not\in I}]$
obtained from $g$ by specializing $x_i = 0$ for every $i\in I$,
and $\bff_I$ the associated family of polynomials
$$\bff_I=((f_j)_I)_{j \in J_I}.$$ Then, $\bff_I$ is the set of polynomials obtained by
specializing the variables indexed by $I$ to $0$ in the
polynomials in $\bff$ and discarding the ones that vanish identically, and $\cA^I = (\cA_j^I)_{j\in J_I}$ is the family of supports of $\bff_I$.

We will use an auxiliary polynomial system defined as follows:
$$\bff(I) = (f_{1,I},\dots, f_{n,I}), \ \hbox{where} \ f_{j,I} = \begin{cases} (f_j)_I & \hbox{ if } j\in J_I \\ f_j & \hbox{ if } j\notin J_I \end{cases}.$$
Note that the family of supports of these polynomials is
$$\cA(I) = (\cA_{1,I},\dots, \cA_{n,I}), \ \hbox{where} \ \cA_{j,I} = \begin{cases} \cA_j^I & \hbox{ if } j\in J_I \\ \cA_j & \hbox{ if } j\notin J_I \end{cases}.$$

\begin{lemma} \label{lem:multfI}
Under the previous assumptions and notation, if $\zeta \in \C^n$ is an isolated zero of $\bff$ lying in $O_I$, then $\zeta$ is also an isolated zero of $\bff(I)$ and $\mbox{mult}_\zeta(\bff) \le \mbox{mult}_\zeta(\bff (I))$.
\end{lemma}

\begin{proof}
The fact that $\zeta$  is an isolated zero of $\bff(I)$ follows from the facts that $\bff(I)$ is a generic system supported on $\cA(I)$ vanishing at $\zeta$, and that, for every $\widetilde I \subset I$, $J_{\widetilde I } (\cA(I)) = J_{\widetilde I}(\cA)$. The inequality between the multiplicities is a consequence of Lemma \ref{lem:multgen}.
\end{proof}

We now focus on a special case of polynomial systems with the same structure as $\bff(I)$, namely,  systems of $n$ polynomials in $n$ variables which contain $r$ polynomials depending only on $r$ variables.

\begin{proposition} \label{prop:multh}
Let $\bfh = (h_1,\dots, h_n)$ be a system of polynomials in $\C[x_1,\dots, x_n]$ such that $h_1,\dots, h_r \in \C[x_1,\dots, x_r]$.
Let  $\xi \in \C^r$ be an isolated nondegenerate common zero of $h_1,\dots h_r$ such that $0\in \C^{n-r}$ is an isolated zero of $\bfh_\xi:= (h_{r+1}(\xi, x_{r+1},\dots, x_n), \dots,$ $ h_{n}(\xi, x_{r+1},\dots, x_n))$. Then, $\zeta= (\xi, 0)\in \C^n$ is an isolated zero of $\bfh$ satisfying:
$$\mult_\zeta (\bfh) = \mult_0(\bfh_\xi).$$
\end{proposition}

\begin{proof}
Under our assumptions, it follows that $\zeta = (\xi , 0)$ is an isolated zero of the system $\bfh$: if there is an irreducible curve $C$ passing through $\zeta$, since $\xi$ is an isolated common zero of $h_1,\dots, h_r\in \C[x_1,\dots, x_r]$, we have that $C \subset \{ x_1 = \xi_1, \dots, x_r = \xi _r\}$ and so, $(\xi, 0) \in C \subset \{ x_1 = \xi_1, \dots, x_r = \xi _r, h_{r+1}(x) = 0\dots, h_n(x) = 0\} = \{ \xi\} \times V(\bfh_\xi)$, contradicting the fact that $0$ is an isolated zero of $\bfh_\xi$.

In order to prove the stated equality of multiplicities, we will compare the multiplicity matrices $S_k(\bfh, \zeta)$ and $S_k(\bfh_\xi, 0)$ for $k\in \N$ (see Section \ref{sec:DZ} for the definition of multiplicity matrices). To this end, we will analyze the structure of $S_k(\bfh, \zeta)$.

Recall that for the system $\bfh$, for $k \ge 1$, the columns of $S_k(\bfh, \zeta)$ are indexed by $\alpha$ for
$|\alpha|\le k$ and its rows are indexed by $(\beta,j)$ for
$|\beta|\le k-1$ and $1\le j \le n$; the entry corresponding to row $(\beta, j)$ and column $\alpha$ is $$(S_k(\bfh, \zeta))_{(\beta, j), \alpha} =\partial_\alpha ((x-\zeta)^{\beta} h_j)(\zeta),$$
where $\partial_\alpha$ is defined in (\ref{eq:deralpha}).

Note that, for  $\gamma=(\gamma_1, \dots, \gamma_n) \in (\Z_{\ge 0})^n$,  we have
\begin{equation} \label{eq:derivada}
\frac{1}{\alpha!}\frac{\partial^{|\alpha|}}{\partial
x^\alpha}((x-\zeta)^{\beta} x^\gamma)(\zeta)=
\begin{cases}
\prod\limits_{i=1}^r \binom{\gamma_i}{\alpha_i-\beta_i}
\xi_i^{\gamma_i+\beta_i-\alpha_i} & \mbox{ if } \beta_i \le \alpha_i \le \beta_i +\gamma_i \ \forall \, 1\le i \le r \\[-2mm]
& \mbox{ and } \alpha_i = \beta_i +\gamma_i \ \forall\,  r+1\le i \le n, \\[2mm]
0 & \mbox{ otherwise.}
\end{cases}\end{equation}

Then, an entry  of $S_k(\bfh, \zeta)$ corresponding to a row indexed by $(\beta, j)$ and a column indexed by $\alpha$ is $0$ whenever $|\beta| \ge |\alpha|$.

We will first consider the columns of  $S_k(\bfh, \zeta)$ indexed by vectors of the form $\alpha= (0,\dots, 0, \alpha_{r+1},\dots, \alpha_n)\ne 0$. For $1\le j \le r$, since the polynomial $h_j$ does not depend on the variables $x_{r+1},\dots, x_n$, we have that $(S_k(\bfh, \zeta))_{(\beta, j), \alpha} =0$ for every $\beta$. For $r+1\le j \le n$ and $\beta$ with $\beta_i \ne 0$ for some $1\le i \le r$, we also have $(S_k(\bfh, \zeta))_{(\beta, j), \alpha} =0$ since $\beta_i >\alpha_i = 0$ (see equation (\ref{eq:derivada})). Finally, for $r+1\le j \le n$ and $\beta = (0,\dots, 0, \beta_{r+1},\dots,\beta_n)$,
\begin{equation}\label{eq:matrixentries}
\begin{array}{rcl}
(S_k(\bfh, \zeta))_{(\beta, j), \alpha} &=& \displaystyle\dfrac{1}{\alpha!}\frac{\partial^{|\alpha|}}{\partial x_{r+1}^{\alpha_{r+1}} \dots \partial x_n^{\alpha_n}}
x_{r+1}^{\beta_{r+1}}\dots x_n^{\beta_n} h_j(\xi, x_{r+1},\dots, x_n) (0) \\[5mm] &=& (S_k(\bfh_\xi, 0))_{((\beta_{r+1},\dots, \beta_n), j), (\alpha_{r+1},\dots, \alpha_n)}.
\end{array}
\end{equation}

We analyze now the remaining columns of the matrix.

Consider the submatrix of $S_k(\bfh, \zeta)$ given by the columns indexed by $\alpha$ such that $(\alpha_1,\dots, \alpha_r ) \ne 0$  and
$|\alpha|=k$.
From identity (\ref{eq:derivada}), we can observe that in every row indexed by $(\beta, j)$ for
$|\beta|=k-1$ and $1\le j \le n$, the only columns with (possibly) non-zero
coordinates are indexed by $\alpha=\beta +
e_i$ where $\{e_i\}_{i=1}^n$ is the canonical basis of $\R^n$;
moreover,
$$(S_k(\bfh, \zeta))_{(\beta, j), \beta + e_i} =
\frac{\partial h_j}{\partial x_i}(\zeta).$$
Note that, for $1 \le j \le r$ and $r+1\le i \le n$, we have $\frac{\partial h_j}{\partial
x_i}\equiv 0 $.
Then, for every $\beta$ with $|\beta|=k-1$, in the rows indexed by $(\beta, j)$
for $ 1 \le j \le r$ we have a copy  of the Jacobian matrix $\mathcal{J} := \left(\frac{\partial h_j}{\partial
x_i}(\xi)\right)_{1\le j,i \le r}$ in the columns indexed by $\beta +e_1,\dots, \beta+e_r$, and
all other entries of the matrix $S_k(\bfh, \zeta)$ in these rows are zero. We remark that $\mathcal{J}$ is an invertible matrix since $\xi$ is a nonsingular common zero of $h_1,\dots h_r$. Note that, for every
$\alpha$ with $|\alpha|=k$ and $\alpha_i\ge 1$ for some $1 \le i \le r$,
there is at least one  $\beta=\alpha-e_i$ with $|\beta|=k-1$; so, all the columns indexed by $\alpha$ with $|\alpha|=k$ and  $(\alpha_1,\dots, \alpha_r ) \ne 0$ are involved in at least one of the copies of $\mathcal{J}$.

Therefore, by performing row operations in $S_k(\bfh, \zeta)$ we can obtain a matrix such that each column indexed by a vector $\alpha$ with $|\alpha|=k$ and $(\alpha_1,\dots, \alpha_r ) \ne 0$ contains all zero entries except for a unique coordinate equal to $1$ in a row indexed by $(\beta, j)$ for some $\beta$ with $|\beta |=k-1$ and $1\le j \le r$, and all these $1$'s lie in different rows. Moreover, these row operations do not modify the remaining columns of $S_k(\bfh, \zeta)$.

Then, the dimension of the kernel of $S_k(\bfh,\zeta)$ is the same as the dimension of the kernel of the matrix obtained by removing the columns indexed by $\alpha$ with $(\alpha_1,\dots, \alpha_r) \ne 0$  and $|\alpha|=k$.
We repeat this procedure for $s=k, k-1, \dots, 1$ (in this order) and we conclude that
the dimension of the kernel of $S_k(\bfh,\zeta)$ is the same as the dimension of the kernel of the submatrix obtained by removing all columns indexed by $\alpha$ with $(\alpha_1,\dots, \alpha_r) \ne 0$. This submatrix consists of the first column of  $S_k(\bfh,\zeta)$, which is identically zero, and all columns indexed by $\alpha= (0,\dots, 0, \alpha_{r+1},\dots,\alpha_n)\ne 0$.
Due to our previous considerations on the matrix formed by these columns, we have that the only rows that are not zero are those indexed by $(\beta, j)$ with $r+1\le j \le n$ and $\beta= (0,\dots, 0,\beta_{r+1}, \dots, \beta_n)$ and these are exactly the rows of $S_k(\bfh_\xi, 0)$ (see identity (\ref{eq:matrixentries})). Therefore,
$$\dim(\ker(S_k(\bfh, \zeta)))=\dim(\ker(S_k(\bfh_\xi, 0))) \mbox{ for
every } k \ge 1.$$ The result follows.
\end{proof}

Now we can prove Theorem \ref{teo:multF=multG}.

\begin{proof}
Without loss of generality, we may assume that $I=\{r+1, \dots, n\}$ for some
$r\in \{1, \dots, n\}$ and $J_I=\{1, \dots, r\}$.

We will first prove that  $\mult_{\zeta}(\bff)\ge \mult_0(\bfg)$.

We make the change of variables
$$ \begin{matrix}
x_1 := \sum_{i=1}^r c_{1i}y_i + \zeta_1 & \quad & x_{r+1} := y_{r+1} \\
\vdots & & \vdots \\
x_r := \sum_{i=1}^r c_{ri}y_i + \zeta_r & \quad &     x_n := y_n
\end{matrix}$$
where $(c_{ki})_{1 \le k,i \le r}\subset \Q$ are
generic constants and obtain the polynomial system
$\widetilde{\bff}=(\widetilde{f}_1,\dots, \widetilde{f}_n)$ in $\C[y_1,\dots, y_n]$ from the system $\bff$. Note that $\mbox{mult}_0(\widetilde{\bff}) = \mbox{mult}_{\zeta}(\bff)$.

For every $1\le j \le n$, let $\widetilde \cA_j$ be the support of $\widetilde f_j$.

For $1\le j \le r$, since
$f_j(x_1, \dots, x_r, 0, \dots, 0)\ne 0$ and has a non-constant term (since it vanishes at $(\zeta_1,\dots, \zeta_r)\in (\C^*)^r$),
due to the genericity of the coefficients and the change of variables, we have that the monomials $y_1, \dots, y_r$ appear with non-zero coefficients in $\widetilde{f}_j(y)$.
On the other hand, again, for the genericity of coefficients and change of variables, for $r+1 \le j \le n$, \begin{equation}\label{eq:supp1}
\pi_I(\widetilde{\cA}_j) = \pi_I( \cA_j);
\end{equation}
moreover, taking into account that $$\widetilde f_j(0, \dots, 0, y_{r+1},\dots, y_n) = f_j(\zeta_1, \dots, \zeta_r, x_{r+1},\dots, x_n),$$ we conclude that
\begin{equation}\label{eq:supp2}
\{ \beta \in (\Z_{\ge 0 })^{n-r}\mid (\mathbf{0}, \beta) \in \widetilde{\cA}_j) \} = \pi_I(\cA_j).
\end{equation}
Let $\mathbf{h}= (h_1, \dots, h_n)$ be a generic polynomial system with supports $\widetilde \cA = (\widetilde{\cA}_1,\dots,\widetilde{\cA}_n)$. Note that condition (H1) holds for $\widetilde \cA$. Let us see that $\widetilde \cA$ also satisfies condition (H2), which implies that $0$ is an isolated zero of $\bfh$. For $\widetilde I \subset \{1,\dots, n\}$, if $\#\widetilde{I} +  \# J_{\widetilde I}(\widetilde \cA) <n$, when setting $y_i = 0$ in $\widetilde \bff$ for every $i\in \widetilde I$, we obtain a system in $n-\#\widetilde I$ unknowns with $\# J_{\widetilde I}(\widetilde \cA)<n-\#\widetilde I$ equations. This system vanishes at $0$ and defines a positive dimensional variety, contradicting the fact that $0\in \C^n$ is an isolated common zero of $\widetilde \bff$.
By Lemma \ref{lem:multgen}, the inequality
$\mbox{mult}_0 (\widetilde{\bff}) \ge \mbox{mult}_0(\mathbf{h})$ holds.

Applying Proposition \ref{prop:sacar monomios} to the polynomials in the system $\mathbf{h}$,
since the monomials $y_1,\dots, y_r$ appear with non-zero coefficients in $h_1,\dots, h_r$ and, for $r+1\le j \le n$, the supports $\mbox{supp}(h_j) = \mbox{supp}(\widetilde f_j)$ satisfy conditions (\ref{eq:supp1}) and (\ref{eq:supp2}), it follows that
$\mbox{mult}_0(\mathbf{h}) =
\mbox{mult}_0(\widetilde {\mathbf{g}},\mathbf{g})$, where $\widetilde {\mathbf{g}} = (\widetilde g_{1},\dots, \widetilde g_r)$  with $\widetilde g_j = \sum_{i=1}^r\vartheta_{ji} y_i + p_j(y_{r+1}, \dots,
y_n)$  for  $1 \le j \le r$, and $\mathbf{g} = (g_{r+1},\dots, g_n)$ with
  $g_j\in \C[y_{r+1},\dots, y_n]$ a generic polynomial with support $\pi_I(\cA_j)$  for $r+1 \le j \le n$.

 Then, if $A$ is the inverse of the matrix $(\vartheta_{ji})$ and $$A .\, (\widetilde g_1, \dots, \widetilde  g_r)^t = (y_1 + q_1(y_{r+1},\dots, y_n), \dots, y_r + q_r(y_{r+1},\dots, y_n))^t,$$ the following is an isomorphism:
$$\begin{array}{ccl}
\Q[y_1, \dots, y_n]/{(\widetilde {\mathbf{g}},\mathbf{g})} &\rightarrow& \Q[y_{r+1},
\dots, y_n]/{( \mathbf{g})}\\[2mm]
\overline{y_i} & \mapsto & \overline{-q_i} \quad \mbox{ for all
}1 \le i \le r \\ \overline{y_i} &  \mapsto & \overline{y_i}
\quad \mbox{ for all }r+1 \le i \le n \end{array}$$
and hence
$\mbox{mult}_0(\widetilde {\mathbf{g}},\mathbf{g})= \mbox{mult}_0(\mathbf{g})
.$

Therefore,
$$ \mbox{mult}_{\zeta}(\bff)= \mbox{mult}_0(\widetilde{\bff}) \ge \mbox{mult}_0(\mathbf{h}) = \mbox{mult}_0(\widetilde {\mathbf{g}},\mathbf{g})= \mbox{mult}_0(\mathbf{g}).$$

To prove the other inequality, note that, by Lemma \ref{lem:multfI}, we have that
 $$\mult_\zeta(\bff) \le \mult_\zeta(\bff(I)).$$
Then, applying Proposition \ref{prop:multh} to
the system $\bff(I)$ and $\xi = (\zeta_1,\dots, \zeta_r)$, we deduce that
$$\mult_\zeta(\bff(I)) = \mult_0(\bff(I)_\xi).$$
By the genericity of the coefficients of $\bff$ and the triangular structure of $\bff(I)$, the system  $\bff(I)_\xi$ turns to be a generic system supported on $\cB_{r+1}^I, \dots, \cB_n^I$.

We conclude that
$\mult_\zeta(\bff) \le \mbox{mult}_0(\mathbf{g}).$
\end{proof}

Taking into account that the results in Section \ref{sec:mult0} enable us to express the multiplicity of the origin as an isolated zero of a generic sparse system in terms of mixed volumes and mixed integrals, we can now state a similar result regarding the multiplicity of any affine isolated zero of a generic sparse system of $n$ equations in $n$ unknowns.

\begin{theorem} \label{teo:MultXiTODO} Let $\cA = (\cA_1, \dots, \cA_n)$ be a family of finite sets in
$(\Z_{\ge 0})^n$ and $\bff = (f_1, \dots, f_n)$ be a generic
sparse system of polynomials in $\C[x_1, \dots, x_n]$ supported on $\cA$.
Let $I \subset \{1, \dots, n\}$ satisfying conditions (A1), (A2) and (A3).
For $j\notin J_I$, let $\cB_j^I= \pi_I(\cA_j)$, where  $\pi_I: \Z^n \to \Z^{\# I}$ is the projection to the
coordinates indexed by $I$.  Let $M_I:=MV_{\#I}\big((\cB_j^I \cup \{ 0 \})_{j \not\in J_I}\big)-
MV_{\#I}\big((\cB_j^I)_{j \not\in J_I}\big)+1$.

Then, for every isolated zero $\zeta \in \C^n$ of $\bff$ such that $\zeta_i = 0$ if and only if $i\in I$, we have
$$\mult_{\zeta}(\bff)= MV_{\#I}((\cB_j^I \cup \{ 0, M_I e_i\}_{i=1}^{\#I})_{j\notin J_I}) -
MV_{\#I}((\cB_j^I \cup \{M_I e_i\}_{i=1}^{\#I})_{j\notin J_I}).
$$
Moreover, if $(\rho_j)_{j \not\in J_I}$ are the convex functions that parameterize the lower
envelopes of the polytopes
$\conv(\cB_j^I \cup \{M_I e_i\}_{i=1}^{\#I})$ and $(\overline{\rho}_j)_{j\notin J_I}$ are their restrictions as defined in (\ref{eq:restricciones}), then $\mult_{\zeta}(\bff) = MI'_{\# I}((\overline{\rho}_j)_{j\notin J_I})$.

\end{theorem}

Note that the previous formula for multiplicities can be refined applying Proposition \ref{prop:refined} instead of Proposition \ref{prop:Mgrande}.

\subsection{Examples}

The following examples illustrate the result in the previous section.

\begin{example} Consider the generic polynomial system
$$ \left\{\begin{array}{l}
 c_{11} x_1^2 + c_{12}x_1^2x_2^2 + c_{13} x_1x_3 +c_{14} x_1x_2^2x_3 + c_{15} x_3^4 + c_{16}x_2^2x_3^4=0\\
 c_{21} x_1^4 + c_{22}x_1^4x_2^2 + c_{23} x_1^2x_3 +c_{24} x_1^2x_2^2x_3 + c_{25} x_3^4 + c_{26}x_2^2x_3^4=0\\
 c_{31} x_1 + c_{32}x_1x_2^2 + c_{33} +c_{34} x_2^2 + c_{35} x_3 + c_{36}x_2^2x_3=0\\
\end{array}\right.
$$
taken from \cite[Example 3]{HJS10}. There is a unique nonempty set $I = \{ 1, 3\}$ satisfying conditions (A1), (A2) and (A3), which leads to two isolated solutions with  $x_1=0$, $x_2
\ne 0$ and $x_3 = 0$. Since $J_I=\{3\}$, Theorem
\ref{teo:multF=multG} tell us that the multiplicity of each of these solutions equals the multiplicity of $(0,0)$ as an isolated root of a generic sparse system supported on $\mathcal{B}^I_1 = \{ (2,0), (1,1), (0,4)\}$ and $\mathcal{B}^I_2= \{ (4,0), (2,1),(0,4)\}$, namely a system of the type
$$\left\{\begin{array}{l} a_1x_1^2 + b_1x_1x_3 + c_1x_3^4 = 0 \\
a_2x_1^4 + b_2x_1^2x_3 + c_3x_3^4=0
\end{array}\right.$$

This multiplicity can be computed, by Proposition \ref{prop:mult0}, as $MV_2(\mathcal{B}^I_1\cup\{( 0,0) \}, \mathcal{B}^I_2\cup\{(0,0)\}) - MV_2(\mathcal{B}^I_1 , \mathcal{B}^I_2 ) = 7$ or, alternatively, by Theorem \ref{teo:MIcasoTocandoEjes}, as
$MI_2'(\overline{\rho}_1, \overline{\rho}_2)=7$, where
$\overline{\rho}_1$ and $\overline{\rho}_2$ are the functions whose graphs are given in Example \ref{ex2}.
\end{example}

\begin{example} Consider the generic polynomial system
$$\left\{\begin{array}{l}
 a_{11}x_1 + a_{12}x_1x_2 =0 \\
 a_{21}x_2^2 + a_{22}x_1^2x_2^4 + a_{23}x_1^3 =0\\
 a_{31}x_3 + a_{32}x_1x_3 + a_{33}x_3^2x_4^2 + a_{34}x_3^3x_4 =0\\
 a_{41}x_4^3 + a_{42}x_2^3x_4^3 + a_{43}x_3^2x_4^3 + a_{44}x_4^5 + a_{45}x_3^2x_4^5 =0\\
\end{array}\right.
$$
Using \cite[Proposition 5]{HJS13} we can check that all zeros of
the system are isolated. Moreover, all the subsets $I\subset \{1,2,3,4\}$ satisfying conditions (A1), (A2) and (A3) are
$$I_1=\emptyset, \ \ I_2 = \{3\}, \ \ I_3 = \{1,2\}, \ \ I_4 = \{3,4\}, \ \ I_5 = \{1,2,3\}
\mbox{ and } I_6 = \{1,2,3,4\}.$$
By Bernstein's theorem, the system has  24 different simple zeros with all non-zero coordinates
(associated to $I_1$) and, by Theorem \ref{teo:MultXiTODO}, we can see that there are
\begin{itemize}
\item $6$ simple zeros associated to $I_2$,
\item $8$ zeros with multiplicity $2$ associated to $I_3$,
\item $3$ zeros with multiplicity $3$ associated to $I_4$,
\item $2$ zeros with multiplicity $2$ associated to $I_5$,
\end{itemize}
and that the origin is an isolated zero of multiplicity $6$.

That is, the system has a total of $65$ (isolated) zeros counting multiplicities. Note that, in this case, $SM_4(\cA) =65$ is smaller than $MV_4(\cA\cup\{0\})= 85$.
\end{example}


\begin{thebibliography}{00}

\bibitem{Ber75} D. N. Bernstein, The number of roots of a system of
equations. Funct. Anal. Appl. 9 (1975), 183--185.

\bibitem{BS09} P. B\"urgisser, P. Scheiblechner,
On the complexity of counting components of algebraic varieties.
J. Symbolic Comput.  44 (2009), no. 9, 1114-–1136.

\bibitem{CD07} E. Cattani, A. Dickenstein,
Counting solutions to binomial complete intersections.
J. Complexity  23 (2007), no. 1, 82-–107.

\bibitem{CLO} D. Cox, J. Little, D. O'Shea, Using Algebraic Geometry.
Grad. Texts in Math., vol. 185. Springer, New York, 1998.

\bibitem{CD06} M.A. Cueto, A. Dickenstein, Some results on
inhomogeneous discriminants. Proceedings of the XVIth Latin
American Algebra Colloquium, Bibl. Rev. Mat. Iberoamericana, Madrid, 2007, 41--62.

\bibitem{DZ05}  B.H. Dayton, Z. Zeng, Computing the multiplicity structure
in solving polynomial systems. Proc. 2005 Internat. Symp.
Symbolic and Algebraic Computation, ACM, New York, 2005, 116--123.

\bibitem{EV99} I.Z. Emiris, J. Verschelde, How to count efficiently all affine roots of a polynomial system. In: 13th European Workshop
on Computational Geometry CG’97. W\"urzburg, 1997, Discrete Appl. Math. 93 (1999), no. 1, 21--32.

\bibitem{GKZ94} I.M. Gelfand, M.M. Kapranov, A.V. Zelevinsky, Discriminants, resultants, and multidimensional determinants. Mathematics: Theory \& Applications. Birkh\"auser Boston, Inc., Boston, MA, 1994.

\bibitem{HJS10} M.I. Herrero, G. Jeronimo, J. Sabia, Computing isolated roots of sparse polynomial systems in affine space.  Theoret. Comput. Sci.  411 (2010),  no. 44-46, 3894--3904.

\bibitem{HJS13} M.I. Herrero, G. Jeronimo, J.  Sabia, Affine solution sets of
sparse polynomial systems. J. Symbolic Comput. 51 (2013), 34--54.

\bibitem{HS97} B. Huber, B. Sturmfels, Bernstein's theorem in affine space.
Discrete Comput. Geom. 17 (1997), no. 2, 137--141.

\bibitem{KK14} K. Kaveh, A.G. Khovanskii, Convex bodies and multiplicities
of ideals. Proc. Steklov Inst. Math.  286 (2014), no. 1, 268--284.

\bibitem{Kho78} A.G. Khovanskii, Newton polyhedra and toroidal varieties. Funct. Anal. Appl. 11 (1978), 289-–296.

\bibitem{Kou76}  A.G. Kouchnirenko, Poly\`edres de Newton et nombres de Milnor. Invent. Math.  32 (1976), no. 1, 1--31.

\bibitem{Kus76}  A.G. Kushnirenko, Newton polytopes and the B\'ezout theorem. Funct. Anal. Appl. 10 (1976), 233--235.

\bibitem{LW96}  T.Y. Li, X. Wang, The BKK root count in $\mathbb{C}^n$.
Math.  Comp. 65 (1996), no. 216, 1477--1484.

\bibitem{Macaulay1916}  F.S. Macaulay, The algebraic theory of modular
systems. Cambridge Univ. Press., Cambridge, 1916.

\bibitem{Mondal16}  P. Mondal, Intersection multiplicity, Milnor number and Bernstein's theorem. Preprint. arXiv:1607.04860

\bibitem{Oka97}  M. Oka, Non-degenerate complete intersection singularity. Actualit\'es Math\'ematiques.
Hermann, Paris, 1997.

\bibitem{PS08a} P. Philippon, M. Sombra, Hauteur normalis\'ee
des vari\'et\'es toriques projectives. J. Inst. Math. Jussieu 7
 (2008), no. 2, 327--373.

\bibitem{PS08b} P. Philippon, M. Sombra,: A refinement of the
Bernstein-Kushnirenko estimate. Adv. Math. 218 (2008), no. 5,
1370--1418.

\bibitem{Roj94} J.M. Rojas, A convex geometrical approach to counting the roots of a polynomial system. Theoret. Comput. Sci. 133 (1994), no. 1, 105--140.

\bibitem{RW96} J.M. Rojas, X. Wang, Counting affine roots of polynomial
systems via pointed Newton polytopes. J. Complexity 12 (1996),
no. 2, 116--133.

\bibitem{Stetter04} H.J. Stetter, Numerical polynomial algebra. SIAM, Philadelphia, 2004.

\bibitem{Tei88} B. Teissier, Mon\^omes, volumes et multiplicit\'es.   Introduction \`a la th\'eorie des singularit\'es, II,  127--141, Travaux en Cours, 37, Hermann, Paris, 1988.

\end{thebibliography}
\end{document}